\documentclass[11pt]{amsart}
\usepackage{latexsym}
\usepackage{fullpage}
\usepackage{amssymb}
\usepackage{amscd}
\usepackage{float}
\usepackage{graphicx}
\usepackage{amsfonts}
\usepackage{pb-diagram}
 \usepackage{amsmath,amscd}

\newtheorem{thm}{Theorem}
\newtheorem{cor}{Corollary}
\newtheorem{lemma}{Lemma}
\newtheorem{prop}{Proposition}
\newtheorem{defn}{Definition}
\newtheorem{remark}{Remark}

\newtheorem{ex}{Example}
\newtheorem{nt}{Notation}

\begin{document}

\title[The Kauffman bracket skein module of the lens spaces via unoriented braids]
  {The Kauffman bracket skein module of the lens spaces via unoriented braids}
	
\author{Ioannis Diamantis}
\address{Department of Data Analytics and Digitalisation,
Maastricht University, School of Business and Economics,
P.O.Box 616, 6200 MD, Maastricht,
The Netherlands.
}
\email{i.diamantis@maastrichtuniversity.nl}

\keywords{ skein module, Kauffman bracket, unoriented braids, solid torus, lens spaces, mixed links, mixed braids, braid group of type B, generalized Hecke algebra of type B, generalized Temperley-Lieb algebra of type B}

\setcounter{section}{-1}

\date{}

\begin{abstract}
In this paper we develop a braid theoretic approach for computing the Kauffman bracket skein module of the lens spaces $L(p,q)$, KBSM($L(p,q)$), for $q\neq 0$. For doing this, we introduce a new concept, that of an {\it unoriented braid}. Unoriented braids are obtained from standard braids by ignoring the natural top-to-bottom orientation of the strands. We first define the {\it generalized Temperley-Lieb algebra of type B}, $TL_{1, n}$, which is related to the knot theory of the solid torus ST, and we obtain the universal Kauffman bracket type invariant, $V$, for knots and links in ST, via a unique Markov trace constructed on $TL_{1, n}$. The universal invariant $V$ is equivalent to the KBSM(ST).  For passing now to the KBSM($L(p,q)$), we impose on $V$ relations coming from the band moves (or slide moves), that is, moves that reflect isotopy in $L(p,q)$ but not in ST, and which reflect the surgery description of $L(p,q)$, obtaining thus, an infinite system of equations. By construction, solving this infinite system of equations is equivalent to computing KBSM($L(p,q)$). We first present the solution for the case $q=1$, which corresponds to obtaining a new basis, $\mathcal{B}_{p}$, for KBSM($L(p,1)$) with $(\lfloor p/2 \rfloor +1)$ elements. We note that the basis $\mathcal{B}_{p}$ is different from the one obtained by Hoste \& Przytycki. For dealing with the complexity of the infinite system for the case $q>1$, we first show how the new basis $\mathcal{B}_{p}$ of KBSM($L(p,1)$) can be obtained using a diagrammatic approach based on unoriented braids, and we finally extend our result to the case $q>1$. The advantage of the braid theoretic approach that we propose for computing skein modules of c.c.o. 3-manifolds, is that the use of braids provides more control on the isotopies of knots and links in the manifolds, and much of the diagrammatic complexity is absorbed into the proofs of the algebraic statements.

\smallbreak
\bigbreak

\noindent 2020 {\it Mathematics Subject Classification.} 57K31, 57K14, 20F36, 20F38, 57K10, 57K12, 57K45, 57K35, 57K99, 20C08.

\end{abstract}

\maketitle
	
\setcounter{tocdepth}{1}
\tableofcontents

\section{Introduction}\label{intro}

Skein modules were introduced independently by Przytycki \cite{P} and Turaev \cite{Tu} as generalizations of knot polynomials in $S^3$ to knot polynomials in arbitrary 3-manifolds. They are quotients of free modules over isotopy classes of links in 3-manifolds by properly chosen local (skein) relations. A skein module of a 3-manifold yields all possible isotopy invariants of knot which satisfy a particular skein relation. Skein modules based on the Kauffman bracket skein relation 
\[
L_+-AL_{0}-A^{-1}L_{\infty}
\] 
\noindent where $L_{\infty}$ and $L_{0}$ are represented schematically by the illustrations in Figure~\ref{skein}, are called {\it Kauffman bracket skein modules} (KBSM). 

\smallbreak

For example, the Kauffman bracket skein module of $S^3$, KBSM($S^3$), is freely generated by the unknot and it is equivalent to the Kauffman bracket for framed links in $S^3$ up to regular isotopy. It's ambient isotopy counterpart is the well-known Jones polynomial. Recall that the algebraic counterpart construction of the Kauffman bracket is the Temperley-Lieb algebra together with a unique Markov trace constructed by V.F.R. Jones (see \cite{Jo} and references therein). It is worth pointing out the importance of the Temperley-Lieb algebras. Through the pioneering work of V.F.R. Jones (\cite{Jo, Jo1}), these algebras related knot theory to statistical mechanics, topological quantum field theories and the construction of quantum invariants for 3-manifolds (works of Witten, Reshetikhin-Turaev, Lickorish, etc).

\bigbreak

The precise definition of KBSM is as follows:

\begin{defn}\rm
Let $M$ be an oriented $3$-manifold and $\mathcal{L}_{{\rm fr}}$ be the set of isotopy classes of unoriented framed links in $M$. Let $R=\mathbb{Z}[A^{\pm1}]$ be the Laurent polynomials in $A$ and let $R\mathcal{L}_{{\rm fr}}$ be the free $R$-module generated by $\mathcal{L}_{{\rm fr}}$. Let $\mathcal{S}$ be the ideal generated by the skein expressions $L_+-AL_{0}-A^{-1}L_{\infty}$ and $L \bigsqcup {\rm O} - (-A^2-A^{-2})L$. Note that blackboard framing is assumed and that $L \bigsqcup {\rm O}$ stands for the union of a link $L$ and the trivially framed unknot in a ball disjoint from $L$. 

\begin{figure}[H]
\begin{center}
\includegraphics[width=2.1in]{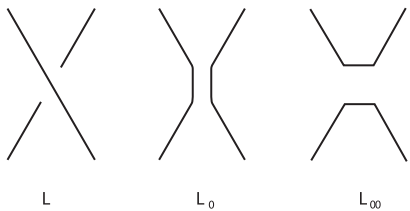}
\end{center}
\caption{The links $L$, $L_{0}$ and $L_{\infty}$ locally.}
\label{skein}
\end{figure}

\noindent Then the {\it Kauffman bracket skein module} of $M$, KBSM$(M)$, is defined to be:

\begin{equation*}
{\rm KBSM} \left(M\right)={\raise0.7ex\hbox{$
R\mathcal{L}_{{\rm fr}} $}\!\mathord{\left/ {\vphantom {R\mathcal{L_{{\rm fr}}} {\mathcal{S} }}} \right. \kern-\nulldelimiterspace}\!\lower0.7ex\hbox{$ S  $}}.
\end{equation*}

\end{defn}

Skein modules of $3$-manifolds have become very important algebraic tools in the study of $3$-manifolds, since their properties renders topological information about the $3$-manifolds. In this paper we are interested in the Kauffman bracket skein module of the lens spaces $L(p,q)$, for $p\neq 0$ using braid theoretical tools. We consider the lens spaces $L(p,q)$ obtained from $S^3$ by rational surgery along the unknot with coefficient $p/q$. Surgery along the unknot is realized by considering the complementary solid torus and attaching to it a solid torus according to a $(p,q)$-homeomorphism on the boundary. Hence, the Kauffman bracket skein module of the solid torus, KBSM(ST), is our starting point. KBSM(ST) is essential in the study of Kauffman bracket skein modules of arbitrary c.c.o. $3$-manifolds, since every c.c.o. $3$-manifold can be obtained by surgery along a framed link in $S^3$ with unknotted components. The family of the lens spaces, $L(p,q)$, comprises the simplest example, since, as mentioned before, they are obtained by rational surgery on the unknot. For a survey on skein modules see \cite{P}.

\bigbreak

From the above it is clear that the knot theory of ST is of fundamental importance for studying knot theory in other c.c.o. 3-manifolds. A basis for KBSM(ST) is presented in \cite{Tu} using diagrammatic methods (see also \cite{HK}). More precisely:

\begin{thm}[\cite{Tu}]
The Kauffman bracket skein module of ST, KBSM(ST), is freely generated by an infinite set of generators $\left\{x^n\right\}_{n=0}^{\infty}$, where $x^n$ denotes $n$ parallel copies of a longitude of ST and $x^0$ is the affine unknot (see Figure~\ref{tur}).
\end{thm}

\begin{figure}[H]
\begin{center}
\includegraphics[width=5in]{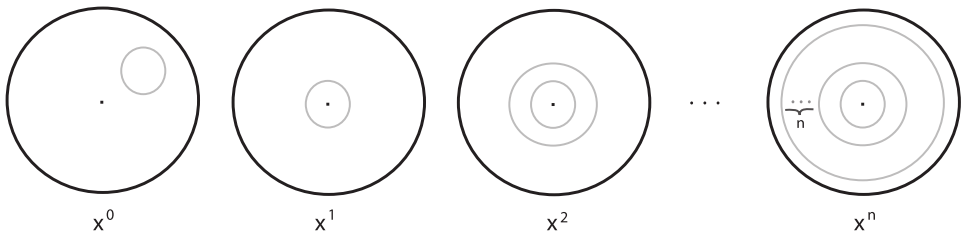}
\end{center}
\caption{The Turaev basis of KBSM(ST).}
\label{tur}
\end{figure}

In this paper we develop a braid theoretic approach to the Kauffman bracket skein module of $L(p,q)$, for $q\neq 0$. Note that the KBSM($L(p,q)$) has been computed via a diagrammatic approach in \cite{HP}, where it is shown that: 

\begin{thm}[\cite{HP}, Theorem 4]
The Kauffman bracket skein module of the lens spaces $L(p, q)$ for $p\geq 1$ is freely generated by $\{x^i\}_{i=0}^{\lfloor p/2 \rfloor}$, where $\lfloor p/2 \rfloor$ denotes the integer part of $p/2$.
\end{thm}

We propose a new and more efficient method for computing Kauffman bracket skein modules via the notion of {\it unoriented braids}, which we define along the way. The algebraic techniques developed in \cite{LR1, LR2, La1, La2, DL1, DL2, DLP, DL4, D1, D2, D3, D4} are of great importance. More precisely, in \cite{La2} it is shown that the generalized Hecke algebra of type B, ${H}_{1,n}$, is related to the knot theory of the solid torus and the Artin group of the Coxeter group of type B, $B_{1, n}$. Moreover, the most generic analogue of the HOMFLYPT polynomial, $X$, for links in the solid torus ST has been derived from ${H}_{1,n}$ via a unique Markov trace constructed on it. The invariant $X$ recovers the HOMFLYPT skein module of ST (see \cite{La2, DL2, DL3}) and as shown in \cite{DLP, DL4}, it can be extended to an invariant for knots in $L(p,1)$ by considering the effect of the {\it braid band moves}, i.e. the analogue of band moves on the braid level (see \cite{LR1, LR2, DL1}) on elements in some basis of the HOMFLYPT skein module of ST. As shown in \cite{P}, this is also true for the case of the Kauffman bracket skein module of $L(p,q),\, p\geq 1$. That is, in order to compute KBSM($L(p,q)$), we consider elements in some basis of KBSM(ST) and impose on the universal invariant for knots in ST, relations coming from the band moves. That is:

\begin{equation}
{\rm KBSM}\left(L(p, q)\right)=\frac{{\rm KBSM}({\rm ST})}{<a-bbm(a)>}, \quad {\rm where}\ a\ {\rm basis\ element\ of\ KBSM(ST)}.
\end{equation}

In this paper we give the complete analogue of KBSM($L(p,q)$) via braids. We first define the {\rm generalized Temperley-Lieb algebra} of type B, $TL_{1, n}$ as the quotient of ${H}_{1,n}$ over the ideal generated by expressions of the form: 

\begin{equation}\label{ideal}
\begin{array}{lcl}
\sigma_{i, i+1} & := & 1 + u\ (\sigma_i+ \sigma_{i+1}) + u^2\ (\sigma_i\sigma_{i+1}+\sigma_{i+1}\sigma_i)+u^3\ \sigma_i\sigma_{i+1}\sigma_i.\\
\end{array}
\end{equation}   

\noindent That is, $TL_{1, n}\, :=\, \frac{H_{1, n}}{<\sigma_{i, i+1}>}$. We then show that the trace function constructed on $H_{1, n}$ (\cite{La2}) factors through $TL_{1, n}$ and, by normalizing it, we obtain the universal Kauffman bracket, $V$, for knots and links in ST. We then extend this invariant to knots and links in $L(p,q)$ for $p\geq 1$, by solving the infinite system of equations resulting from the band moves. Namely, we force:

\begin{equation}\label{eqbbm}
V_{\widehat{a}}\ =\  V_{\widehat{bm(a)}},
\end{equation}

\noindent for all $a$ in the basis of KBSM(ST), where $bm(a)$ denotes the result of a band move on $a$.

\smallbreak

Equations~(\ref{eqbbm}) have particularly simple formulations with the use of a new basis, $B_{{\rm ST}}$, for the Kauffman bracket skein module of ST, that is presented first in Section~2 (see Eq.~(\ref{basis}) in this paper), in terms of mixed braids (that is, classical braids with the first strand identically fixed). We prove that $B_{{\rm ST}}$ is a basic set of KBSM(ST) by relating the braided form of the Turaev basis, ${B}_{{\rm ST}}^{\prime}$, to ${B}_{{\rm ST}}$ via a lower triangular matrix with invertible elements in the diagonal. For an illustration of elements in the basis ${B}_{{\rm ST}}$ see the bottom of Figure~\ref{allbases}. Note that the same basis is presented in \cite{D1} via the Tempreley-Lieb algebra of type B and in \cite{GM1} via a diagrammatic approach based on arrow diagrams. We then solve the infinite system of Eq.~(\ref{eqbbm}) for the case of $L(p,1)$ and we arrive at the main result of this paper. In particular, we have the following result:

\begin{thm}\label{important}
The Kauffman bracket skein module of the lens spaces $L(p,1)$, $p\geq 1$, is freely generated by elements in the set $\mathcal{B}_p$, consisting of element in the form $\{t^i\}_{i=0}^{\lfloor p/2 \rfloor}$. For an illustration of elements in $\mathcal{B}_p$ see Figure~\ref{allbases}.
\end{thm}

Note that the basis $\mathcal{B}_p$ is different from the basis presented by Hoste and Przytycki in \cite{HP} for KBSM($L(p,1)$). After we establish the algebraic method for the case $q=1$, we present a diagrammatic method based on braids and we use it to derive the basis $\mathcal{B}_p$ for KBSM($L(p,1)$). We then extend this diagrammatic method for the case of $L(p,q)$, $p, q\geq 1$, and in particular, we show that the set $\mathcal{B}_p$ forms a basis for KBSM($L(p,q)$). We finally discuss the algebraic method for the case of KBSM($L(p,q)$).

\bigbreak

The paper is organized as follows: In \S\ref{basics} we recall the setting and the essential techniques and results from \cite{La1, La2, LR1, LR2, DL1}. More precisely, we present isotopy moves for knots and links in $L(p,q)$ and we then describe braid equivalence for knots and links in $L(p,q)$. In \S\ref{SolidTorus} we present results from \cite{La2} and \cite{FG} and we extend these results in order to construct the universal invariant for knots in ST of the Kauffman bracket type. More precisely, we start by recalling results on the generalized Hecke algebra of type B, $H_{1, n}$, and we present the universal invariant of the HOMFLYPT type for knots in ST via a unique Markov trace, $tr$, defined on the algebra. This invariant captures the HOMFLYPT skein module of ST. We then pass to the generalized Temperley-Lieb algebra of type B, defined as a quotient algebra from $H_{1, n}$ over the ideal generated by relations (\ref{ideal}), and we find necessary and sufficient conditions so that the trace $tr$ factors through $TL_{1, n}$. Using the trace, we define the universal invariant $V$ for knots in ST of the Kauffman bracket type. In \S~\ref{SMSST} we present results on KBSM(ST) via braids and we present a new basis, $B_{\rm ST}$, for KBSM(ST) via an ordering relation defined in \cite{DL2}. This basis is also presented in \cite{D1} using the Temperley-Lieb algebra of type B, which, roughly speaking, is related to an invariant for knots in ST, but not the most generic one that captures KBSM(ST). Moreover, the results presented for obtaining $B_{\rm ST}$ are essential for the computation of the Kauffman bracket skein module of the lens spaces $L(p,1)$, $p\geq 1$, that we compute in \S~\ref{kblens} using algebraic techniques. This is done with the use of the newly introduced concept of unoriented braids defined in \S~\ref{unbr1}. These new objects seem promising for studying knots in c.c.o. 3-manifolds (and c.c.o. 3-manifolds in general) via the proposed technique. Finally, in \S~\ref{diag} we present the basis $\mathcal{B}_p$ for $L(p,1)$ using a diagrammatic approach based on unoriented braids and in \S~\ref{lpq} we show how these results can be generalized for the case $L(p,q), q>1$.

\bigbreak

It is worth mentioning that the importance of our approach lies in the fact that it can shed light to the problem of computing (various) skein modules of arbitrary c.c.o. $3$-manifolds (see \cite{D4} for the case of the Kauffman bracket skein module of the complement of $(2, 2p+1)$-torus knots). The main difficulty of the problem lies in solving the infinite system of equations (\ref{eqbbm}).

\bigbreak

\noindent \textbf{Acknowledgments}\ \ I would like to acknowledge several discussions with Professor Sofia Lambropoulou, who provided insight and expertise that greatly assisted this research.

\section{Preliminaries}\label{basics}

\subsection{Mixed links and isotopy in $L(p,q)$}

In this section we recall results from \cite{LR1, LR2, DL1}. We shall consider ST to be the complement of a solid torus in $S^3$. Then, an oriented link $L$ in ST can be represented by an oriented \textit{mixed link} in $S^{3}$, that is, a link in $S^{3}$ consisting of the unknotted fixed part $\widehat{I}$ representing the complementary solid torus in $S^3$, and the moving part $L$ that links with $\widehat{I}$. A \textit{mixed link diagram} is a diagram $\widehat{I}\cup \widetilde{L}$ of $\widehat{I}\cup L$ on the plane of $\widehat{I}$, where this plane is equipped with the top-to-bottom direction of $I$ (for an illustration see Figure~\ref{mli}).

\begin{figure}[H]
\begin{center}
\includegraphics[width=3.1in]{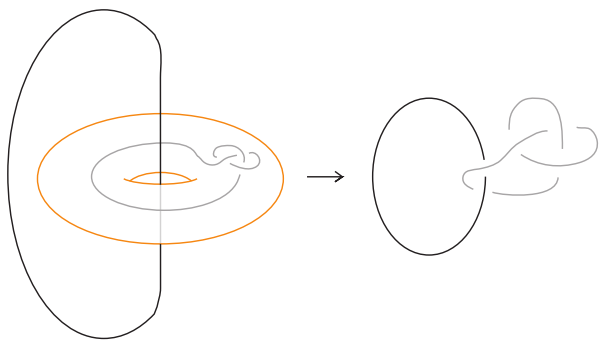}
\end{center}
\caption{ A mixed link. }
\label{mli}
\end{figure}

\smallbreak

The lens spaces $L(p,q)$ can be obtained from $S^3$ by surgery on the unknot with surgery coefficient $p\in \mathbb{Z}$. Surgery along the unknot can be realized by considering first the complementary solid torus and then attaching to it a solid torus according to some homeomorphism $h$ on the boundary. For $L(p,1)$, the homeomorphism $h$ on the boundaries of the solid tori maps the meridian, $m_1$, of one solid torus to a $(p,1)$-curve on the other, that is, a $(p\cdot l_2+1\cdot m_2)$-curve, where $l_2$ denotes the longitude and $m_2$ denotes the meridian of the second solid torus:
\[
\begin{matrix}
h & : & \partial\, {\rm ST}_1 & \rightarrow & \partial\, {\rm ST}_2\\
  &   &   m_1                  & \mapsto     &  p\cdot l_2+1\cdot m_2
\end{matrix}
\]

Thus, isotopy in $L(p,q)$ can be viewed as isotopy in ST together with the {\it band moves} in $S^3$, which reflect the surgery description of the manifold. In Figure~\ref{bmov}, the two types of band moves for the case of $L(p, 1)$ are illustrated, according to the orientation of the component of the knot and of the surgery curve. In the $\alpha$-type the orientation of the arc is opposite to the orientation of the surgery curve, but after the performance of the move their orientations agree. In the $\beta$-type the orientations initially agree, but disagree after the performance of the move. For the case of $L(p, q)$ see Figure~\ref{bbm1}. In \cite{DL1} it is shown that in order to describe isotopy for knots and links in a c.c.o. $3$-manifold, it suffices to consider only one type of band moves (cf. Thm.~6 \cite{DL1}) and thus, isotopy between oriented links in $L(p,q)$ is reflected in $S^3$ by means of the following theorem:

\begin{thm}
Two oriented links in $L(p,q)$ are isotopic if and only if two corresponding mixed link diagrams of theirs differ by isotopy in {\rm ST} together with a finite sequence of one type of band moves.
\end{thm}

\begin{figure}[H]
\begin{center}
\includegraphics[width=5in]{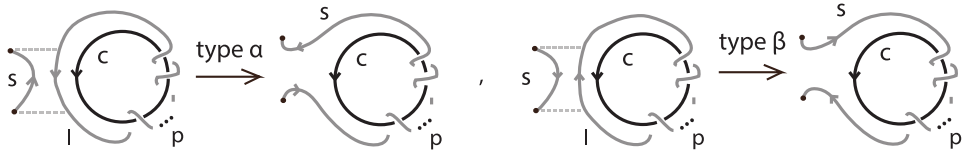}
\end{center}
\caption{ The two types of band moves for $q=1$. }
\label{bmov}
\end{figure}

\subsection{Mixed braids and braid equivalence for knots and links in $L(p, q)$}

In this subsection we introduce the notion of mixed braids, the mixed braid group $B_{1, n}$ and braid equivalence in $L(p, q)$. Note that a mixed braid in $S^{3}$ is a braid where, without loss of generality, its first strand represents $\widehat{I}$, the fixed part, and the other strands, $\beta$, represent the moving part $L$. We shall call the subbraid $\beta$ the \textit{moving part} of $I\cup \beta$ (see bottom left hand side of Figure~\ref{bmov}). We now recall the analogue of the Alexander theorem for knots in ST (cf. Thm.~1 \cite{La2}):

\begin{thm}[{\bf The analogue of the Alexander theorem for ST}]
A mixed link diagram $\widehat{I}\cup \widetilde{L}$ of $\widehat{I}\cup L$ may be turned into a \textit{mixed braid} $I\cup \beta$ with isotopic closure.
\end{thm}

In order to translate isotopy for links in $L(p, q)$ into braid equivalence, we first perform the technique of {\it standard parting} introduced in \cite{LR2} in order to separate the moving strands from the fixed strand that represents the lens spaces $L(p, q)$. This can be realized by pulling each pair of corresponding moving strands to the right and {\it over\/} or {\it under\/} the fixed strand that lies on their right. Then, we define a {\it braid band move} to be a move between mixed braids, which is a $\alpha$-type band move between their closures. It starts with a little band oriented downward, which, before sliding along a surgery strand, gets one twist {\it positive\/} or {\it negative\/} (see bottom of Figure~\ref{bbmov} for $q=1$ and Figure~\ref{bbm2} for $q>1$). For now we shall only consider the $\alpha$-type band moves, since the result of an $\alpha$-type band move remain braided, while the result of a $\beta$-band move does not. The $\beta$-band moves will become important later, when we compute KBSM($L(p, q)$) diagrammatically. Note also that there are two different types of braid band moves; the positive and the negative braid band move, depending on the kind of the twist the moving strand gets before sliding along the fixed strand.

\begin{figure}
\begin{center}
\includegraphics[width=3.4in]{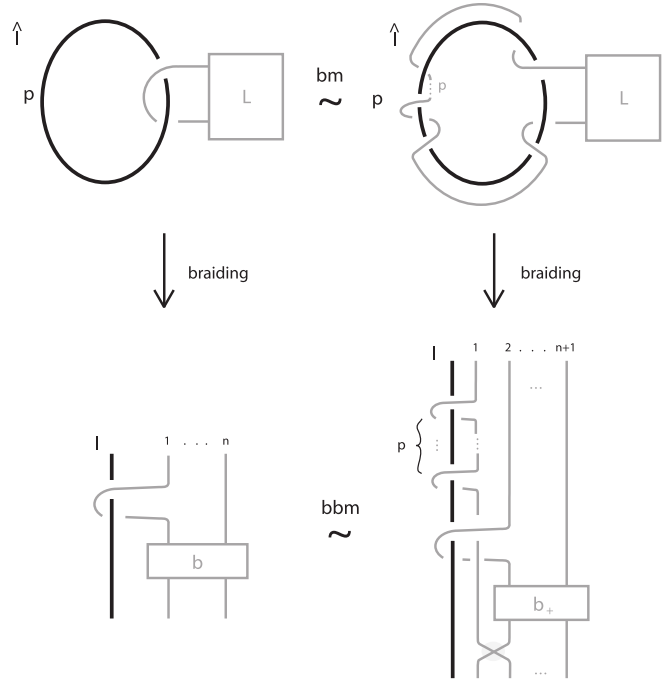}
\end{center}
\caption{ Isotopy in $L(p,1)$ and the two types of braid band moves on mixed braids.}
\label{bbmov}
\end{figure}

\begin{figure}
\begin{center}
\includegraphics[width=2.5in]{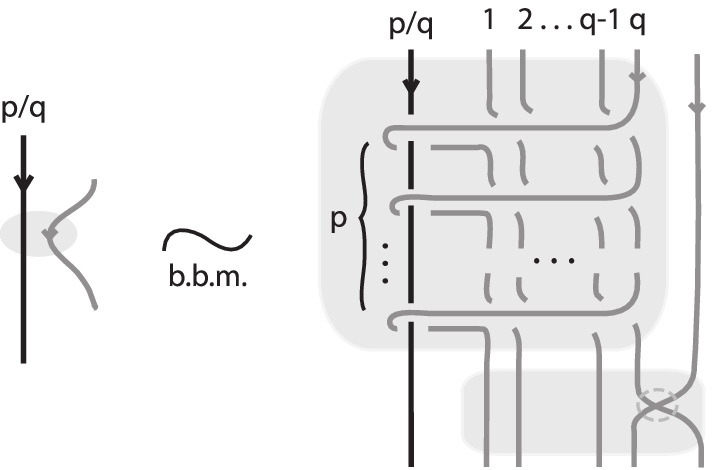}
\end{center}
\caption{The braid band move for $L(p,q),\, q>1$.}
\label{bbm2}
\end{figure}

The sets of braids related to ST form groups, which are in fact the Artin braid groups of type B, denoted $B_{1,n}$, with presentation:

\[ B_{1,n} = \left< \begin{array}{ll}  \begin{array}{l} t, \sigma_{1}, \ldots ,\sigma_{n-1}  \\ \end{array} & \left| \begin{array}{l}
\sigma_{1}t\sigma_{1}t=t\sigma_{1}t\sigma_{1} \ \   \\
 t\sigma_{i}=\sigma_{i}t, \quad{i>1}  \\
{\sigma_i}\sigma_{i+1}{\sigma_i}=\sigma_{i+1}{\sigma_i}\sigma_{i+1}, \quad{ 1 \leq i \leq n-2}   \\
 {\sigma_i}{\sigma_j}={\sigma_j}{\sigma_i}, \quad{|i-j|>1}  \\
\end{array} \right.  \end{array} \right>, \]

\noindent where the generators $\sigma _{i}$ and $t$ are illustrated in Figure~\ref{genh}(i).

\smallbreak

Let now $\mathcal{L}$ denote the set of oriented knots and links in ST. Isotopy in $L(p, q)$ is then translated on the level of mixed braids by means of the following theorem:

\begin{thm}[Theorem~5, \cite{LR2}] \label{markov}
 Let $L_{1} ,L_{2}$ be two oriented links in $L(p,1)$ and let $I\cup \beta_{1} ,{\rm \; }I\cup \beta_{2}$ be two corresponding mixed braids in $S^{3}$. Then $L_{1}$ is isotopic to $L_{2}$ in $L(p,1)$ if and only if $I\cup \beta_{1}$ is equivalent to $I\cup \beta_{2}$ in $\mathcal{B}$ by the following moves:
\[ \begin{array}{clll}
(i)  & Conjugation:         & \alpha \sim \beta^{-1} \alpha \beta, & {\rm if}\ \alpha ,\beta \in B_{1,n}. \\
(ii) & Stabilization\ moves: &  \alpha \sim \alpha \sigma_{n}^{\pm 1} \in B_{1,n+1}, & {\rm if}\ \alpha \in B_{1,n}. \\
(iii) & Loop\ conjugation: & \alpha \sim t^{\pm 1} \alpha t^{\mp 1}, & {\rm if}\ \alpha \in B_{1,n}. \\
(iv) & Braid\ band\ moves: & \alpha \sim {t}^p \alpha_+ \sigma_1^{\pm 1}, & a_+\in B_{1, n+1},\ {\rm for}\ L(p, 1),\\
& &  \beta \sim \left[(\sigma_{q-1}\, \ldots\, \sigma_1)\, t\right]^p\,\beta+ \sigma_q^{\pm 1}, & a_+\in B_{1, n+q},\ {\rm for}\ L(p,q),
\end{array}
\]

\noindent where $\alpha_+$ is the word $\alpha$ with all indices shifted by +1 and $\beta_+$ is the word $\beta$ with all indices shifted by $q$. Note that moves (i), (ii) and (iii) correspond to link isotopy in {\rm ST}.
\end{thm}

\begin{nt}\rm
We denote a braid band move by bbm and, specifically, the result of a positive or negative braid band move performed on a mixed braid $\beta$ by $bbm_{\pm}(\beta)$.
\end{nt}

Note also that in \cite{LR2} it was shown that the choice of the position of connecting the two components after the performance of a bbm is arbitrary.

\section{Knot algebras and invariants of knots in ST}\label{SolidTorus}

In this section we present the most generic invariant, $V$, for knots and links in ST that captures the Kauffman bracket skein module of ST. Therefore, $V$ is the most appropriate invariant to be extended to a Kauffman bracket invariant for knots and links in $L(p, q)$. This invariant is derived via a unique Markov trace on the generalized Temperley-Lieb algebra of type B, $TL_{1, n}$. This algebra is defined in \S~\ref{gtln} and the Markov trace factors through the unique Markov trace on the generalized Hecke algebra of type B. We first recall some results on the generalized Hecke algebra of type B, $H_{1, n}$, since $TL_{1, n}$ is defined as the quotient of $H_{1, n}$ over an appropriate ideal.

\subsection{The generalized Hecke algebra of type B}

In \cite{La2} the most generic analogue of the HOMFLYPT polynomial, $X$, for links in the solid torus ST has been derived from the generalized Hecke algebras of type $\rm B$, $H_{1,n}$, via a unique Markov trace constructed on them. This algebra was defined by Lambropoulou as the quotient of ${\mathbb C}\left[q^{\pm 1} \right]B_{1,n}$ over the quadratic relations ${g_{i}^2=(q-1)g_{i}+q}$. Namely:

\begin{equation*}
H_{1,n}(q)= \frac{{\mathbb C}\left[q^{\pm 1} \right]B_{1,n}}{ \langle \sigma_i^2 -\left(q-1\right)\sigma_i-q \rangle}.
\end{equation*}

In \cite{La2} it is also shown that the following sets form linear bases for ${\rm H}_{1,n}(q)$ (\cite[Proposition~1 \& Theorem~1]{La2}):

\begin{equation}\label{bHeck}
\begin{array}{llll}
 (i) & \Sigma_{n} & = & \{t_{i_{1} } ^{k_{1} } \ldots t_{i_{r}}^{k_{r} } \cdot \sigma \} ,\ {\rm where}\ 0\le i_{1} <\ldots <i_{r} \le n-1,\\
 (ii) & \Sigma^{\prime} _{n} & = & \{ {t^{\prime}_{i_1}}^{k_{1}} \ldots {t^{\prime}_{i_r}}^{k_{r}} \cdot \sigma \} ,\ {\rm where}\ 0\le i_{1} < \ldots <i_{r} \le n-1, \\
\end{array}
\end{equation}
\noindent where $k_{1}, \ldots ,k_{r} \in {\mathbb Z}$, $t_0^{\prime}\ =\ t_0\ :=\ t, \quad t_i^{\prime}\ =\ g_i\ldots g_1tg_1^{-1}\ldots g_i^{-1} \quad {\rm and}\quad t_i\ =\ g_i\ldots g_1tg_1\ldots g_i$ are the `looping elements' in ${\rm H}_{1, n}(q)$ (see Figure~\ref{genh}(ii)) and $\sigma$ a basic element in the Iwahori-Hecke algebra of type A, ${\rm H}_{n}(q)$, for example in the form of the elements in the set \cite{Jo}:

$$ S_n =\left\{(g_{i_{1} }g_{i_{1}-1}\ldots g_{i_{1}-k_{1}})(g_{i_{2} }g_{i_{2}-1 }\ldots g_{i_{2}-k_{2}})\ldots (g_{i_{p} }g_{i_{p}-1 }\ldots g_{i_{p}-k_{p}})\right\}, $$

\noindent for $1\le i_{1}<\ldots <i_{p} \le n-1{\rm \; }$. In \cite{La2} the bases $\Sigma^{\prime}_{n}$ are used for constructing a Markov trace on $\mathcal{H}:=\bigcup_{n=1}^{\infty}\, H_{1, n}$, and using this trace, a universal HOMFLYPT-type invariant for oriented links in ST is constructed.

\begin{figure}[H]
\begin{center}
\includegraphics[width=5.3in]{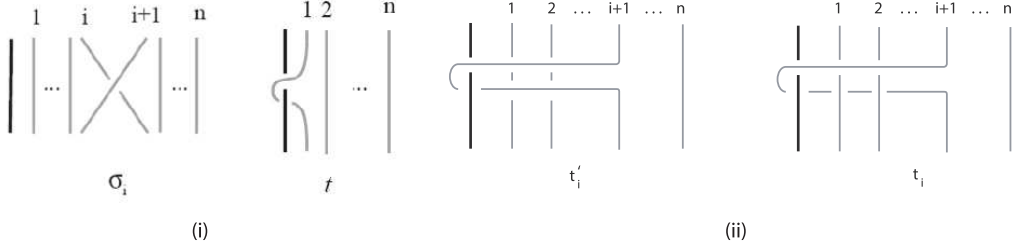}
\end{center}
\caption{The generators of $B_{1, n}$ and the `looping' elements $t^{\prime}_{i}$ and $t_{i}$.}
\label{genh}
\end{figure}

\begin{thm}{\cite[Theorem~6 \& Definition~1]{La2}} \label{tr}
Given $z, s_{k}$ with $k\in {\mathbb Z}$ specified elements in $R={\mathbb C}\left[q^{\pm 1} \right]$, there exists a unique linear Markov trace function on $\mathcal{H}$:

\begin{equation*}
{\rm tr}:\mathcal{H}  \to R\left(z,s_{k} \right),\ k\in {\mathbb Z}
\end{equation*}

\noindent determined by the rules:

\[
\begin{array}{lllll}
(1) & {\rm tr}(ab) & = & {\rm tr}(ba) & \quad {\rm for}\ a,b \in {\rm H}_{1,n}(q) \\
(2) & {\rm tr}(1) & = & 1 & \quad {\rm for\ all}\ {\rm H}_{1,n}(q) \\
(3) & {\rm tr}(ag_{n}) & = & z{\rm tr}(a) & \quad {\rm for}\ a \in {\rm H}_{1,n}(q) \\
(4) & {\rm tr}(a{t^{\prime}_{n}}^{k}) & = & s_{k}{\rm tr}(a) & \quad {\rm for}\ a \in {\rm H}_{1,n}(q),\ k \in {\mathbb Z} \\
\end{array}
\]

\bigbreak

\noindent Then, the function $X:\mathcal{L}$ $\rightarrow R(z,s_{k})$

\begin{equation*}
X_{\widehat{\alpha}} = \Delta^{n-1}\cdot \left(\sqrt{\lambda } \right)^{e}
{\rm tr}\left(\pi \left(\alpha \right)\right),
\end{equation*}

\noindent is an invariant of oriented links in {\rm ST}, where $\Delta:=-\frac{1-\lambda q}{\sqrt{\lambda } \left(1-q\right)}$, $\lambda := \frac{z+1-q}{qz}$, $\alpha \in B_{1,n}$ is a word in the $\sigma _{i}$'s and $t^{\prime}_{i} $'s, $\widehat{\alpha}$ is the closure of $\alpha$, $e$ is the exponent sum of the $\sigma _{i}$'s in $\alpha $, $\pi$ the canonical map of $B_{1,n}$ to ${\rm H}_{1,n}(q)$, such that $t\mapsto t$ and $\sigma _{i} \mapsto g_{i}$.
\end{thm}

\begin{remark}\rm
As shown in \cite{La2, DL2} the invariant $X$ recovers the HOMFLYPT skein module of ST. For a survey on the HOMFLYPT skein module of the lens spaces $L(p, 1)$ via braids, the reader is referred to \cite{DL3, DGLM}.
\end{remark}

\subsection{The generalized Temperley-Lieb algebra of type B}\label{gtln}

We now introduce the analogue of the (normalized) Kauffman bracket polynomial, $V$, for links in the solid torus via the generalized Temperley-Lieb algebra of type B.

\begin{defn}\rm
The {\it generalized Temperley-Lieb algebra of type B}, $TL_{1, n}$, is defined as the quotient of the generalized Hecke algebra of type B, $H_{1,n}(q)$, over the ideal generated by the elements $\sigma_{i, i+1},\, i\in \mathbb{N}\backslash \{0\}$ (recall Eq.~\ref{ideal}).
\end{defn}

\begin{remark}\rm
In \cite{FG} the analogue of the normalized Kauffman bracket polynomial, $V$, for links in the solid torus $\rm ST$ has been derived from the Temperley-Lieb algebra of type B, ${TL}_{n}^{{\rm B}}$. This algebra defined as a quotient of the Hecke algebra of type B, ${H}_{1, n}(q, Q)$, over the ideal generated by the $\sigma_{i, i+1}$ elements and:
\[
\begin{array}{lcl}
h_B & := & 1 + u\ \sigma_1 + v\ t + uv\ (\sigma_1 t + t \sigma_1) + u^2v\ \sigma_1 t \sigma_1 + uv^2\ t \sigma_1 t + (uv)^2\ \sigma_1 t \sigma_1 t
\end{array}
\]
It is also worth mentioning that in \cite{FG} a different presentation for $H_{1, n}$ is used, that involves the parameters $u, v$ and the quadratic relations 
\begin{equation}\label{quad}
{\sigma_{i}^2=(u-u^{-1})\sigma_{i}+1}.
\end{equation}

\noindent One can switch from one presentation to the other by a taking $\sigma_i = u \sigma_i$, $t = v t$ and $q = u^2$. We will adapt this setting from now on.
\end{remark}

Since the generalized Temperley-Lieb algebra of type B is a quotient of the generalized Hecke algebra of type B, we look for necessary and sufficient conditions so that the Markov trace defined in $H_{1, n}$ factors through to ${TL}_{1, n}$. We have the following result:

\begin{thm}\label{tr2}
The trace defined in $H_n(1, q)$ factors through to $TL_{1, n}$ if and only if the trace parameters take the following values:
\begin{equation}\label{parameters}
z=-\frac{1}{u(1+u^2)}.
\end{equation}
\end{thm} 

\begin{proof}
In order to evaluate the values of $z$ so that the trace function factors through ${TL}_{1, n}$, we solve the equation $tr(g_{i, i+1})\, = \, 0$ and we have:
\[
\begin{array}{rclc}
u^3\, tr(g_i\, g_{i+1}\, g_i) & = & -1\, -2u\, z\, -2u^2\, z^2 & \Leftrightarrow\\
&&&\\
u^3\, z\, tr(g_i^2) & = & -1\, -2u\, z\, -2u^2\, z^2 & \Leftrightarrow\\
&&&\\
u^3\, z^2\, (u-u^{-1})\, +\, u^3\, z & = & -1\, -2u\, z\, -2u^2\, z^2 & \Leftrightarrow\\
&&&\\
(uz\, +\, 1)\, (u^3z\, +\, uz\, +\, 1) & = & 0 & \Leftrightarrow\\
&&&\\
z\, =\, -\, \frac{1}{u} & {\rm or} & z\, =\, -\, \frac{1}{u\, (1+u^2)}.  &\\
\end{array}
\]
As explained in \cite{FG}, only the values in (\ref{parameters}) are of topological interest and the proof is now concluded.
\end{proof}

Note now that for $z=-\, \frac{1}{u(1+u^2)}$, one deduces $\lambda\ =\ u^4$. Following \cite{FG}, we obtain the following result:

\begin{thm}\label{inva}
The following invariant is the most generic invariant for knots and links in {\rm ST}:

\begin{equation*}
V^{{\rm B}}_{\widehat{\alpha}}(u ,v)\ :=\ \left(-\frac{1+u^2}{u} \right)^{n-1} \left(u\right)^{2e} {\rm tr}\left(\overline{\pi} \left(\alpha \right)\right),
\end{equation*}

\noindent where $\alpha \in B_{1,n}$ is a word in the $\sigma _{i}$'s and $t^{\prime}_{i} $'s, $\widehat{\alpha}$ is the closure of $\alpha$, $e$ is the exponent sum of the $\sigma _{i}$'s in $\alpha $, $\overline{\pi}$ the canonical map of $B_{1,n}$ to ${\rm TL}_{1, n}$, such that $t\mapsto t$ and $\sigma _{i} \mapsto g_{i}$.
\end{thm}

\section{The Kauffman bracket skein module of ST via braids}\label{SMSST}


In the braid setting, the elements of the Turaev-basis of KBSM(ST) correspond bijectively to the elements of the following set:

\begin{equation}\label{Lpr}
{B}^{\prime}_{{\rm ST}}=\{ t{t^{\prime}_1} \ldots {t^{\prime}_n}, \ n\in \mathbb{N} \}.
\end{equation}

\noindent The set ${B}^{\prime}_{{\rm ST}}$ forms a basis of KBSM(ST) in terms of braids (for an illustration see Figure{allbases}). Note that ${B}^{\prime}_{{\rm ST}}$ is a subset of $\mathcal{H}$ and, in particular, ${B}^{\prime}_{{\rm ST}}$ is a subset of $\Sigma^{\prime}=\bigcup_n\Sigma^{\prime}_n$. Note also that in contrast to elements in $\Sigma^{\prime}$, the elements in ${B}^{\prime}_{{\rm ST}}$ have no gaps in the indices, the exponents are all equal to one and there are no `braiding tails'. 

\smallbreak

The invariant $V$ defined in Theorem~\ref{inva} recovers KBSM(ST). Indeed, it gives distinct values to distinct elements of ${B}^{\prime}_{{\rm ST}}$, since ${\rm tr}(t{t^{\prime}_1} \ldots {t^{\prime}_n})=s_{1}^n$. Hence, in order to compute the Kauffman bracket skein module of the lens spaces $L(p,q)$, it suffices to extend the (most generic) invariant $V$ for knots in $L(p,q)$ following the ideas in \cite{DL2, DL3, DL4}. That is, it suffices to solve the infinite system of Equations~(\ref{eqbbm}), $V_{\widehat{\alpha}}\ =\  V_{\widehat{bbm_{\pm}(\alpha)}}$, for all $\alpha$ in the basis of KBSM(ST). These equations have particularly simple formulations with the use of the new basis, ${B}_{{\rm ST}}$, for the Kauffman bracket skein module of ST first presented in \cite{D1}.

\bigbreak


We now present a different basis $B_{{\rm ST}}$ for the Kauffman bracket skein module of the solid torus, which is crucial toward the computation of ${\rm KBSM}\left(L(p,1)\right)$ and which is described in Eq.~(\ref{basis}) in open braid form. For an illustration see Figure~\ref{allbases}. In particular we have the following:

\begin{thm}\label{newbasis}
The following set is a basis for KBSM(ST):
\begin{equation}\label{basis}
B_{\rm ST}\ =\ \{t^{n},\ n \in \mathbb{N} \}.
\end{equation}
\end{thm}

\begin{figure}
\begin{center}
\includegraphics[width=5in]{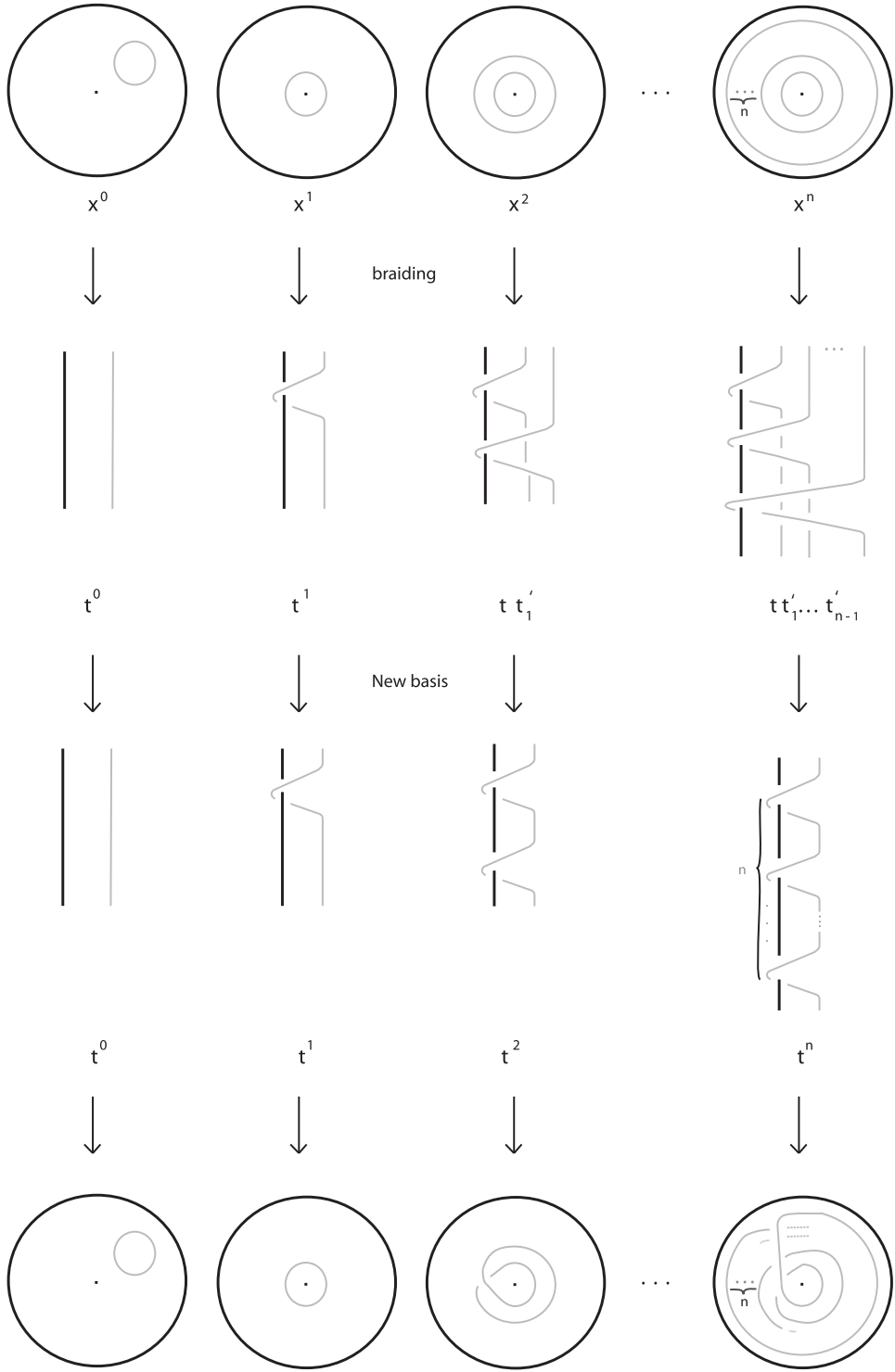}
\end{center}
\caption{Elements in the different bases of KBSM(ST).}
\label{allbases}
\end{figure}

The importance of the new basis $B_{\rm ST}$ of KBSM(ST) lies in the simplicity of the algebraic expression of a braid band move, which extends the link isotopy in ST to link isotopy in $L(p, q)$ (recall Theorem~1(iv)). This was our motivation for establishing this new basis ${B}_{\rm ST}$. Note also that in \cite{D2} the same basis was derived using the Temperley-Lieb algebra of type B. In this section we present a new and different proof for Theorem~\ref{newbasis}. The results presented here will also be used for the computation of KBSM($L(p,1)$) in \S~\ref{kblens}. The method for proving Theorem~\ref{newbasis} is the following:

\smallbreak

\begin{itemize}
\item[$\bullet$] we first recall the total ordering defined in \cite{DL2} for elements in the sets $\Sigma^{\prime}$ and $\Sigma$,
\smallbreak
\item[$\bullet$] we then pass from the set $B^{\prime}_{{\rm ST}}$ consisting of monomials of the form $tt_1^{\prime}t_2^{\prime}\ldots t_n^{\prime},\ n\in \mathbb{N}$, to an augmented set ${B_{{\rm ST}}}^{aug}$, consisting of monomials of the form $t^n$, where $n\in \mathbb{Z}$.
\smallbreak
\item[$\bullet$] We express elements in ${B_{{\rm ST}}}^{aug} \backslash B_{{\rm ST}}$ to sums of elements in $B_{{\rm ST}}$.
\smallbreak
\item[$\bullet$] The two (ordered) sets $B^{\prime}_{\rm ST}$ and $B_{\rm ST}$ are then related via a lower triangular infinite matrix with invertible elements on the diagonal, and
\smallbreak
\item[$\bullet$] we conclude that the set $B_{\rm ST}$ forms a basis of KBSM(ST).
\end{itemize}

Before we proceed with the proof of Theorem~\ref{newbasis}, we introduce the notion of unoriented braids, which is crucial for working with unoriented knots on the ``braid'' level.

\subsection{Unoriented braids}\label{unbr1}

The Kauffman bracket polynomial is defined for unoriented knots, but when working on the braid level, there is a natural orientation from top to bottom. Applying the Kauffman bracket skein relation on a crossing from a standard braid results in merging strands with opposite orientations (see Figure~\ref{uno1}).

\begin{figure}[H]
\begin{center}
\includegraphics[width=2.3in]{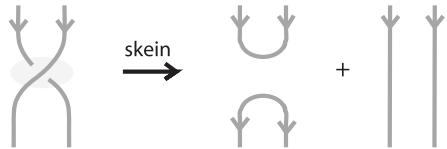}
\end{center}
\caption{ The skein relation on the braid level.}
\label{uno1}
\end{figure}

For this reason, we introduce the notion of {\it unoriented braids} as follows:

\begin{defn}\rm
An {\it unoriented braid} is defined as a classical braid by ignoring the natural top to bottom orientation. 
\end{defn} 

\begin{remark}\rm
The concept of unoriented braids is new and seems promising in serving as a tool for studying knots in c.c.o. 3-manifolds and fore deriving invariants for 3-manifolds in general. For a thorough study of unoriented braids see \cite{DL5}.
\end{remark}

The significance of unoriented braids for the Kauffman bracket skein module of the solid torus (and for Kauffman bracket skein modules in general), is illustrated in Figure~\ref{unb}, where the looping generators $t$ and $t^{-1}$ are shown to be equivalent in KBSM(ST) (compare also with Figure~\ref{nexp}).

\begin{figure}[H]
\begin{center}
\includegraphics[width=3.6in]{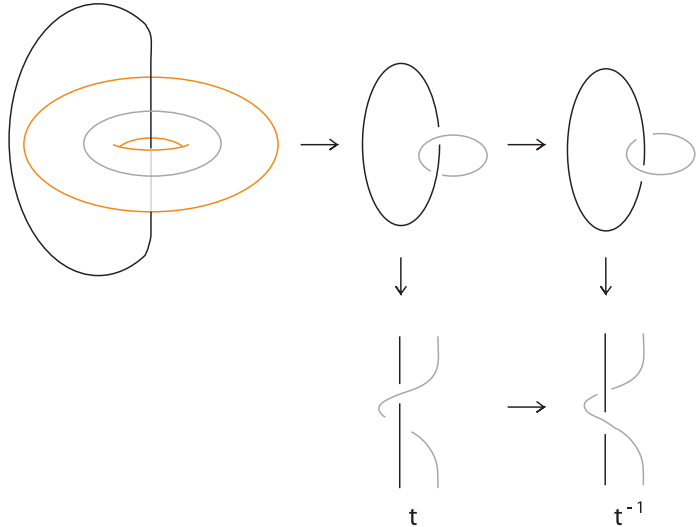}
\end{center}
\caption{Unoriented braids and KBSM(ST).}
\label{unb}
\end{figure}

\smallbreak

From now on, by braids we shall mean unoriented braids unless otherwise stated.

\subsection{An ordering relation}\label{ore}

We now present an ordering relation defined on $\Sigma_n^{\prime}$, $B^{\prime}_{{\rm ST}}$ and $\Sigma_n$. We first introduce the notion of the index of a word in $B^{\prime}_{{\rm ST}}$, in $\Sigma_n^{\prime}$ and in $\Sigma_n$.

\begin{defn} \label{index} \rm
The {\it index} of a word $\tau$ in $B^{\prime}_{{\rm ST}}$ (or in $\Sigma_n^{\prime}$ or in $\Sigma_n$), denoted $ind(\tau)$, is defined to be the highest index of the $t_i^{\prime}$'s (respectively of the $t_i$'s), in $\tau$. Similarly, the \textit{index} of an element in $\Sigma_n^{\prime}$ or in $\Sigma_n$ is defined in the same way by ignoring possible gaps in the indices of the looping generators and by ignoring the braiding part in $\textrm{H}_{n}$. Moreover, the index of a monomial in $\textrm{H}_{n}$ and in $TL_{1, n}$ is equal to $0$.
\end{defn}

\noindent For example, $ind({t^{\prime}}^{k_0}{t^{\prime}_1}^{k_1}\ldots {t^{\prime}_n}^{k_n})\, =\, ind(t^{u_0}\ldots t_n^{u_n})\, =\, n$.

\bigbreak

We now proceed with presenting an ordering relation in the sets $\Sigma$ and $\Sigma^{\prime}$, which passes to their respective subsets $B_{\rm ST}$ and $B_{\rm ST}^{\prime}$.

\begin{defn}{\cite[Definition~2]{DL2}} \label{order}\rm
Let $w={t^{\prime}_{i_1}}^{k_1}\ldots {t^{\prime}_{i_{\mu}}}^{k_{\mu}}\cdot \beta_1$ and $u={t^{\prime}_{j_1}}^{\lambda_1}\ldots {t^{\prime}_{j_{\nu}}}^{\lambda_{\nu}}\cdot \beta_2$ in $\Sigma^{\prime}$, where $k_t , \lambda_s \in \mathbb{Z}$ for all $t,s$ and $\beta_1, \beta_2$ are standard braid words. Then, we define the following ordering in $\Sigma^{\prime}$:

\smallbreak

\begin{itemize}
\item[(a)] If $\sum_{i=0}^{\mu}k_i < \sum_{i=0}^{\nu}\lambda_i$, then $w<u$.

\vspace{.1in}

\item[(b)] If $\sum_{i=0}^{\mu}k_i = \sum_{i=0}^{\nu}\lambda_i$, then:

\vspace{.1in}

\noindent  (i) if $ind(w)<ind(u)$, then $w<u$,

\vspace{.1in}

\noindent  (ii) if $ind(w)=ind(u)$, then:

\vspace{.1in}

\noindent \ \ \ \ ($\alpha$) if $i_1=j_1, \ldots , i_{s-1}=j_{s-1}, i_{s}<j_{s}$, then $w>u$,

\vspace{.1in}

\noindent \ \ \  ($\beta$) if $i_t=j_t$ for all $t$ and $k_{\mu}=\lambda_{\mu}, k_{\mu-1}=\lambda_{\mu-1}, \ldots, k_{i+1}=\lambda_{i+1}, |k_i|<|\lambda_i|$, then $w<u$,

\vspace{.1in}

\noindent \ \ \  ($\gamma$) if $i_t=j_t$ for all $t$ and $k_{\mu}=\lambda_{\mu}, k_{\mu-1}=\lambda_{\mu-1}, \ldots, k_{i+1}=\lambda_{i+1}, |k_i|=|\lambda_i|$ and $k_i>\lambda_i$, then $w<u$,

\vspace{.1in}

\noindent \ \ \ \ ($\delta$) if $i_t=j_t\ \forall t$ and $k_i=\lambda_i$, $\forall i$, then $w=u$.

\end{itemize}

The ordering in the set $\Sigma$ is defined as in $\Sigma^{\prime}$, where $t_i^{\prime}$'s are replaced by $t_i$'s.
\end{defn}

\begin{defn} \label{aug}
\rm
We define the augmented sets $B_{{\rm ST}}^{aug}$ and ${B^{\prime}_{{\rm ST}}}^{aug}$ as follows:
\[
\begin{array}{lcl}
{B_{{\rm ST}}}^{aug} & := & \{t^n,\, n\in \mathbb{Z} \},\\
&&\\
{B^{\prime}_{{\rm ST}}}^{aug} & := & \{t^{k_0}{t^{\prime}_1}^{k_1}\ldots {t^{\prime}_{m}}^{k_m},\, k_{i}\in \mathbb{Z},\, \forall i \}.
\end{array}
\]
\end{defn}

\begin{prop}
The sets ${B_{{\rm ST}}}$, ${B_{{\rm ST}}}^{aug}$, ${B_{{\rm ST}}^{\prime}}$ and ${B_{{\rm ST}}^{\prime}}^{aug}$ equipped with the ordering relation of Definition~\ref{order}, are totally ordered sets. Moreover, the sets ${B_{{\rm ST}}}$ and ${B_{{\rm ST}}^{\prime}}$ are well-ordered sets.
\end{prop}

\begin{proof}
In \cite{DL2}, Proposition~1, it is shown that the sets ${\Sigma_n}^{\prime}$ and ${\Sigma_n}$ equipped with the ordering relation of Definition~\ref{order}, are totally ordered sets. Since ${B_{{\rm ST}}^{\prime}}$ and ${B_{{\rm ST}}^{\prime}}^{aug}$ are subsets of ${\Sigma_n}^{\prime}$, and since ${B_{{\rm ST}}}$ and ${B_{{\rm ST}}}^{aug}$ are subsets of ${\Sigma_n}$, the sets ${B_{{\rm ST}}}$, ${B_{{\rm ST}}}^{aug}$, ${B_{{\rm ST}}^{\prime}}$ and ${B_{{\rm ST}}^{\prime}}^{aug}$ inherit the property of being totally ordered sets from ${\Sigma_n}$ (cor. $\Sigma_n^{\prime}$).
\smallbreak
Moreover, $t^0$ is the minimum element of both ${B_{{\rm ST}}}$ and ${B_{{\rm ST}}^{\prime}}$ and thus, these sets are well-ordered sets.
\end{proof}

\subsection{Useful Lemmata}\label{implem}

In this subsection we prove a series of results in order to convert elements in $B^{\prime}_{\rm ST}$ to elements in $B_{\rm ST}$. We will use the symbol $\widehat{\cong}$ when conjugation and stabilization moves are performed, and $\widehat{\underset{{\rm skein}}{\cong}}$ when both conjugation, stabilization moves and the Kauffman bracket skein relation is performed. Note that in order to simplify our results, we will omit the coefficients involved. Finally, it is worth mentioning that the $t_i^{\prime}$'s are conjugates, i.e. $t_i^{\prime}\, t_j^{\prime}\, \widehat{\cong}\,  t_j^{\prime}\,  t_i^{\prime}$ for all $i \neq j$. A general case is illustrated in Figure~\ref{con}, for $w\in {B^{\prime}_{\rm ST}}^{aug}$.

\begin{figure}
\begin{center}
\includegraphics[width=4.5in]{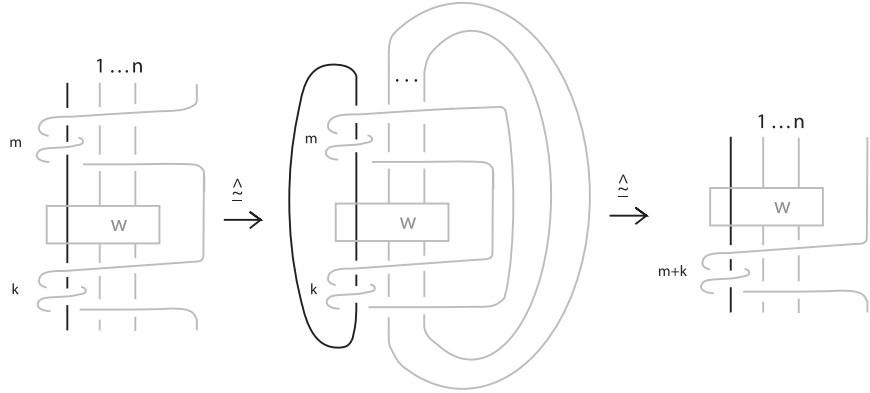}
\end{center}
\caption{Conjugation on the $t_i^{\prime}$'s.}
\label{con}
\end{figure}

\smallbreak

In order to simplify the algebraic expressions obtained throughout this procedure and throughout the paper in general, we first introduce the following notation:

\begin{nt}\label{nt} \rm
We set ${\tau^{\prime}_{i,i+m}}^{k_{i,i+m}}:={t^{\prime}}_i^{k_i}\ldots {t^{\prime}}^{k_{i+m}}_{i+m}\in {B^{\prime}_{{\rm ST}}}^{aug}$, for $m\in \mathbb{N}$, $k_j\neq 0$ for all $j$.
\end{nt}

\begin{lemma}\label{lemst}
Let $w\in {B^{\prime}_{\rm ST}}^{aug}$ such that $ind(w)=n-1$. Then, for $k, m\in \mathbb{Z}$ the following relations hold in KBSM(ST):
\[
w\, {t_{n}^{\prime}}^k\, {t_{n+1}^{\prime}}^m\ \widehat{\underset{{\rm skein}}{\cong}} \ a\cdot w\, {t_{n}^{\prime}}^{k-m}\, +\, b\cdot w\, {t_{n}^{\prime}}^{k+m},
\]
\noindent where $a, b$ coefficients.
\end{lemma}

The proof of Lemma~\ref{lemst} is illustrated in Figure~\ref{lmst}.

\bigbreak

\begin{figure}
\begin{center}
\includegraphics[width=5.5in]{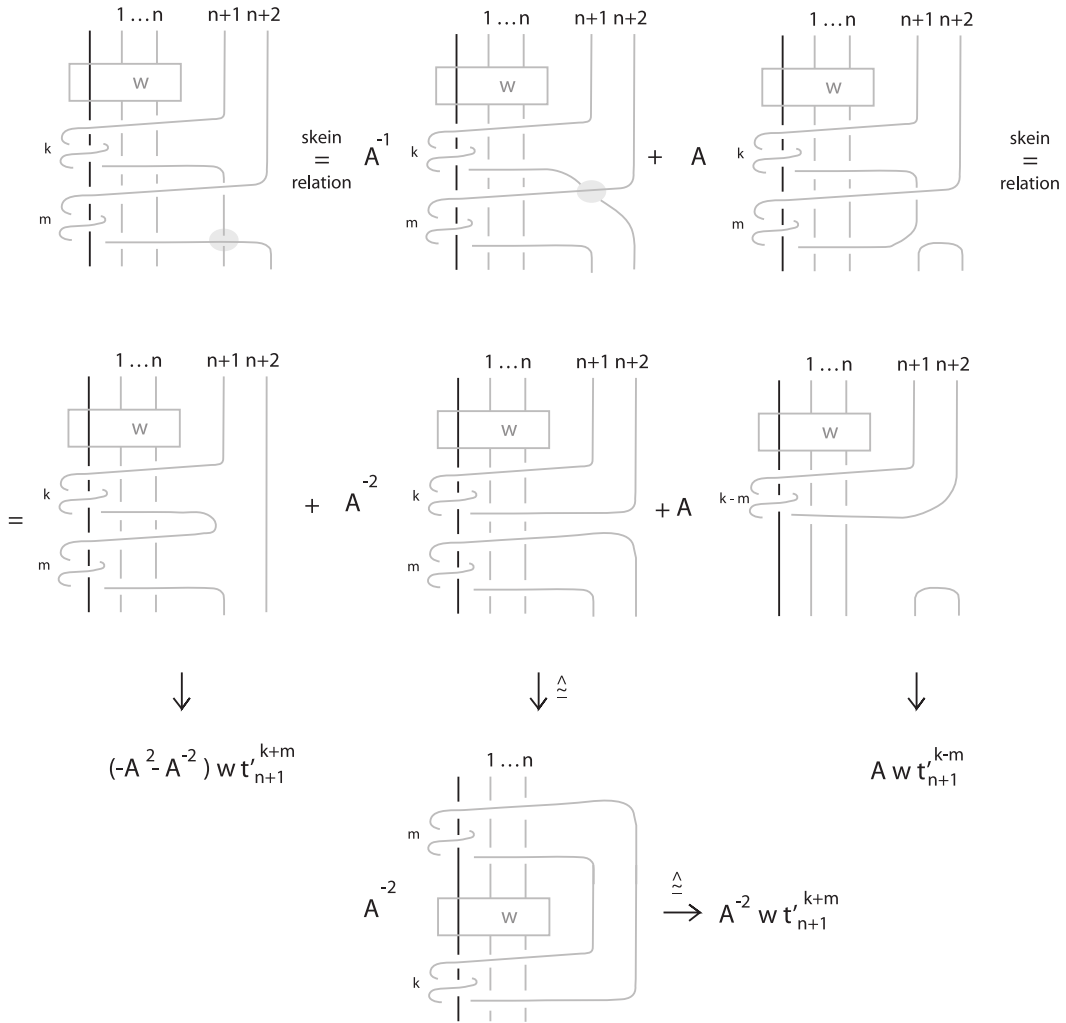}
\end{center}
\caption{The proof of Lemma~\ref{lemst}.}
\label{lmst}
\end{figure}

Note that when applying Lemma~\ref{lemst} on an element in $B^{\prime}_{\rm ST}$, we obtain an element in $B^{\prime}_{\rm ST}$ and an element in ${B^{\prime}_{\rm ST}}^{aug}$ of lower index than the initial element in $B^{\prime}_{\rm ST}$. Indeed we have:
\[
{\tau^{\prime}_{0, n}}\ =\ t_0\, t^{\prime}_1\, \ldots\, t^{\prime}_{n-1}\, t^{\prime}_n\, \overset{Lemma~\ref{lemst}}{\cong}\  t_0\, t^{\prime}_1\, \ldots\, t^{\prime}_{n-2}\ +\ t_0\, t^{\prime}_1\, \ldots\, {t^{\prime}_{n-1}}^2\ =\ {\tau^{\prime}_{0, n-2}}\ +\ {\tau^{\prime}_{0, n-2}}\, {t^{\prime}_{n-1}}^2.
\]

We now convert elements in ${B^{\prime}_{\rm ST}}^{aug}$ to sums of elements in $B_{\rm ST}^{aug}$. We have the following result:

\begin{lemma}\label{lconv}
The following relations hold in KBSM(ST):

\begin{equation*}\label{conv}
{\tau^{\prime}}^{k_{0, n}}_{0, n} \ \widehat{\underset{{\rm skein}}{\cong}} \ \underset{i\leq k}{\sum}\, a_i\, t^i,\\
\end{equation*}

\noindent where $k\, =\, \underset{i=0}{\overset{n}{\sum}}\, k_i$, $k_i\in \mathbb{N}$ and $a_i$ coefficients for all $i$.
\end{lemma}

\begin{proof}
We prove Lemma~\ref{lconv} by strong induction on the order of ${\tau^{\prime}}^{k_{0, n}}_{0, n}$.

\smallbreak

The base of the induction is the monomial $tt_1^{\prime}\in B^{\prime}_{\rm ST}$ of index $1$ and it is illustrated in Figure~\ref{bas1}.

\begin{figure}[H]
\begin{center}
\includegraphics[width=5.5in]{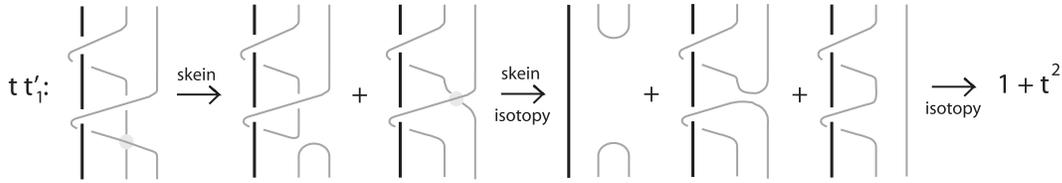}
\end{center}
\caption{The base of induction.}
\label{bas1}
\end{figure}

Assume now that Lemma~\ref{lconv} holds for all monomials of lower order than ${\tau^{\prime}}^{k_{0, n}}_{0, n}$. Then, we have:

\[
{\tau^{\prime}}^{k_{0, n}}_{0, n} \ = \ {\tau^{\prime}}^{k_{0, n-2}}_{0, n-2}\, \underline{{t^{\prime}_{n-1}}^{k_{n-1}}\, {t_{n}^{\prime}}^{k_n}} \ \overset{{\rm Lemma}~\ref{lemst}}{\cong}\ {\tau^{\prime}}^{k_{0, n-2}}_{0, n-2}\, {t_{n-1}^{\prime}}^{k_{n-1}-k_n}\, +\, {\tau^{\prime}}^{k_{0, n-2}}_{0, n-2}\, {t_{n-1}^{\prime}}^{k_{n-1}+k_n}.
\]

According to the ordering relation, on the right hand side of this equation we have two monomials of index $n-1<n=ind({\tau^{\prime}}^{k_{0, n}}_{0, n})$. Thus, these monomials are of lower order than ${\tau^{\prime}}^{k_{0, n}}_{0, n}$. By the induction hypothesis the proof is concluded.
\end{proof}

\begin{remark}\rm
Figure~\ref{lmst} suggests that Lemma~\ref{lconv} is true for $k_i\in \mathbb{Z}$, that is:

\begin{equation}\label{conv1}
{\tau^{\prime}}^{k_{0, n}}_{0, n} \ \widehat{\underset{{\rm skein}}{\cong}} \ \underset{i\leq k}{\sum}\, a_i\, t^i,
\end{equation}
\noindent where $k\, =\, \underset{i=0}{\overset{n}{\sum}}\, |\, k_i\, | $, $k_i\in \mathbb{Z}$ and $a_i$ coefficients for all $i$.
\end{remark}

Applying now Lemma~\ref{conv} (together with Equations~\ref{conv1} when needed) on elements in $B^{\prime}_{{\rm ST}}$ leads to the following result:

\begin{cor}\label{cor1}
Elements in $B^{\prime}_{{\rm ST}}$ can be written as sums of elements in ${B_{{\rm ST}}}^{aug}$. In particular: 
\begin{equation}\label{tprt}
\tau^{\prime}_{0, n}\ \widehat{\underset{{\rm skein}}{\cong}}\ \underset{i\leq n+1}{\sum}\, a_i\, t^i,
\end{equation}
\noindent where $a_i$ coefficients for all $i$. Equivalently, the set ${B_{{\rm ST}}}^{aug}$ spans KBSM(ST).
\end{cor}

In the following example, we demonstrate in detail all calculations involved when applying Lemma~\ref{conv} and Equations~\ref{conv1} to elements in $B^{\prime}_{{\rm ST}}$.

\begin{ex}\label{exa1}\rm
\[
\begin{array}{rclclcl}
t\, t_1^{\prime} & \widehat{\underset{{\rm skein}}{\cong}} & 1\, +\, t^2 &&&&\\
&&&&&&\\
t\, t_1^{\prime}\, t_2^{\prime} & \widehat{\underset{{\rm skein}}{\cong}} & t\, {t^{\prime}_1}^2\, +\, t & \widehat{\underset{{\rm skein}}{\cong}} & t^{-1}\, +\, t\, +\, t^3 &&\\
&&&&&&\\
t\, t_1^{\prime}\, t_2^{\prime}\, t_3^{\prime} & \widehat{\underset{{\rm skein}}{\cong}} & t\, t^{\prime}_1\, {t^{\prime}_2}^2 \, +\, t\, t_1^{\prime} & \widehat{\underset{{\rm skein}}{\cong}} & t\, {t_1^{\prime}}^3\, +\, t\, {t_1^{\prime}}^{-1}\, +\, t\, t_1^{\prime} & \widehat{\underset{{\rm skein}}{\cong}} & t^{-2}\, +\, 1\, +\, t^2\, +\, t^4 \\
&&&&&&\\
t\, t_1^{\prime}\, t_2^{\prime}\, t_3^{\prime}\, t_4^{\prime} & \widehat{\underset{{\rm skein}}{\cong}} & t\, t^{\prime}_1\, {t^{\prime}_2}\, {t^{\prime}_3}^2 \, +\, t\, t_1^{\prime}\, {t^{\prime}_2} & \widehat{\underset{{\rm skein}}{\cong}} & t\, t_1^{\prime}\,  {t_2^{\prime}}^3\, +\, t\, {t_1^{\prime}}\, {t_2^{\prime}}^{-1} & + & t\, t_1^{\prime}\, t_2^{\prime}\\
&&&&&&\\
& \widehat{\underset{{\rm skein}}{\cong}} & t\, {t_1^{\prime}}^{4}\, +\, t\, {t_1^{\prime}}^{-2} & + & t\, +\, t\, {t_1^{\prime}}^2 \, + \, t\, t_1^{\prime}\, t_2^{\prime} & \\
&&&&&&\\
& \widehat{\underset{{\rm skein}}{\cong}} &  t^{-3}\, +\, t^{-1} & + & t\, +\, t^3 & + & t^5\\
\end{array}
\]
\end{ex}

Following the calculations in Example~\ref{exa1}, one may notice that there are closed formulas for Equations~(\ref{tprt}). More precisely, simple calculations lead to the following result:

\begin{lemma}\label{fl1}
For $n\in \mathbb{N}$, the following relations hold in KBSM(ST):
\[
\tau_{0, n}^{\prime} \ \widehat{\underset{{\rm skein}}{\cong}}
\begin{cases}   
\underset{i=0}{\overset{(n+1)/2}{\sum}}\, a_i\, t^{2i} \, + \, \underset{i=1}{\overset{(n-1)/2}{\sum}}\, b_i\, t^{-2i} &, {\rm for}\ n\ {\rm odd}\\
&\\
\underset{i=0}{\overset{n/2}{\sum}}\, c_i\, t^{2i+1} \, + \, \underset{i=1}{\overset{n/2}{\sum}}\, d_i\, t^{-2i+1} &, {\rm for}\ n\ {\rm even}
\end{cases},
\]
\noindent where $a_i, b_i, c_i, d_i$ are coefficients.
\end{lemma}

\subsection{Dealing with negative exponents}\label{ne}

We now deal with monomials in ${B_{{\rm ST}}}^{aug}$ with negative exponents, i.e. elements in ${B_{{\rm ST}}}^{aug}\, \backslash\, B_{{\rm ST}}$, and express them to sums of elements in $B_{{\rm ST}}$. As mentioned before, due to the fact that the Kauffman bracket doesn't take orientation into consideration, we use unoriented braids. In the figures that follow, we start with standard braids and we ignore the orientation when a Kauffman bracket skein relation is performed, resulting in unoriented braids. Note that we will be using the same notation for standard braids and unoriented braids. In Figure~\ref{nexp} the cases of $t^{-1}$ and $t^{-2}$ are illustrated and it is important to note that the resulting monomial $t\, t_1^{\prime}$ is an unoriented braid. Finally, it is worth mentioning that in order to demonstrate the importance of the new proposed method for computing Kauffman bracket skein modules, we will follow a more algebraic procedure for converting elements in ${B_{{\rm ST}}}^{aug}\, \backslash\, B_{{\rm ST}}$ to elements in $B_{{\rm ST}}$ using the Kauffman bracket skein relation, instead of straightforward using the fact that $t^{-n}\, \sim\, t^n$, for all $n\in \mathbb{N}$ on the unoriented braid level.

\begin{figure}[H]
\begin{center}
\includegraphics[width=5.2in]{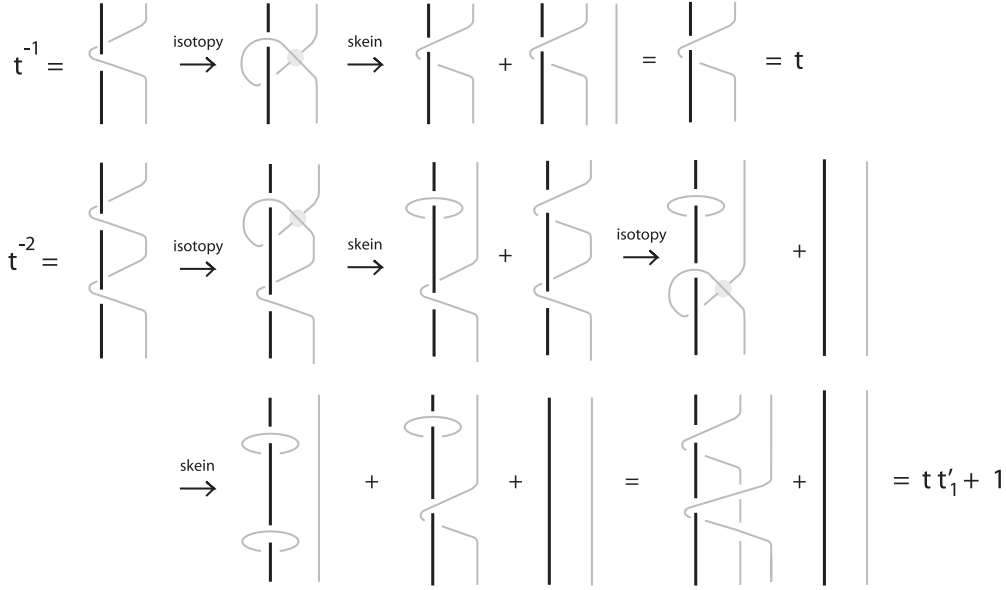}
\end{center}
\caption{Converting $t^{-1}$ and $t^{-2}$ to elements in $B^{\prime}_{{\rm ST}}$.}
\label{nexp}
\end{figure}

Note that in the case of $t^{-2}$, when we apply the skein relation for the first time, a closed looping $\hat{t}$ is ``isolated''. We then apply again the skein relation by ignoring $\hat{t}$ (readers may find useful to think of it as being shrunk). Recall also that $tt_1^{\prime}\, \widehat{\underset{{\rm skein}}{\cong}}\, 1+t^2$ (recall Figure~\ref{bas1}). Thus, $t^{-2}\, \widehat{\underset{{\rm skein}}{\cong}}\, 1\, +\, t^2$ in KBSM(ST).
\bigbreak
Let us consider now the general case for $n\in \mathbb{N}$. We first perform an RI move on the top of the looping generator creating a kink. We then apply the Kauffman bracket skein relation on the resulting crossing and the result is $t^{-n+2}\, +\, t^{-n+1}\, t_1^{\prime}$. We apply the same procedure on the resulting monomials. Note that in the monomial $t^{-n+1}\, t_1^{\prime}$, the looping $t_1^{\prime}$ will not interfere with the above mentioned process on the closure $\widehat{t^{-n+1}\, t_1^{\prime}}$. We continue until all negative exponents become either zero or one. Obviously, this depends on $n$ being odd or even. Hence, we have proved the following result:

\begin{lemma}\label{fl2}
For $n\in \mathbb{N}$, the following relations hold in KBSM(ST):
\[
t^{-n}\ \widehat{\underset{{\rm skein}}{\cong}}\ \begin{cases} \underset{i=0}{\overset{\frac{n-1}{2}}{\sum}}\, a_i\, \tau_{0, 2i}^{\prime} &,\ {\rm for}\ n\ {\rm odd}\\
\underset{i=0}{\overset{\frac{n}{2}}{\sum}}\, a_i^{\prime}\, \tau_{0, 2i}^{\prime} &,\ {\rm for}\ n\ {\rm even}
\end{cases},
\]
\noindent where $a_i, a^{\prime}_i$ coefficients.
\end{lemma}

It is crucial now to note that when we apply Lemma~\ref{fl1} on a monomial $\tau^{\prime}$ in $B_{\rm ST}^{\prime}$, we obtain monomials in $B_{\rm ST}$ and in ${B_{\rm ST}}^{aug} \backslash B_{\rm ST}$. We then apply Lemma~\ref{fl2} on the resulting monomials in ${B_{\rm ST}}^{aug} \backslash B_{\rm ST}$ and obtain elements in $B^{\prime}_{\rm ST}$ again, but of lower order than the initial monomial $\tau^{\prime}$. Thus, if we continue this way, we will eventually reach at elements in $B_{\rm ST}$. Indeed we have the following result:

\begin{prop}\label{pr1}
For $n\in \mathbb{N}$, the following hold in KBSM(ST):
\[
t^{-n}\ \widehat{\underset{{\rm skein}}{\cong}}\ \underset{i=0}{\overset{n}{\sum}}\, a_i\, t^{i},
\]
\noindent where $a_i$ are coefficients for all $i$.
\end{prop}

\begin{proof}
Let $n$ be an odd natural number. The case that $n$ is even follows similarly. We have that:
\[
\begin{array}{lclcl}
t^{-n} & \overset{{\rm Lemma}~\ref{fl2}}{\sim} & \underset{i=0}{\overset{\frac{n-1}{2}}{\sum}}\, a_i\, \tau_{0, 2i}^{\prime} &\overset{{\rm Lemma}~\ref{fl1}}{\sim} & \underset{i=0}{\overset{\frac{n-1}{2}}{\sum}}\, a_i\, \left( \underset{j=0}{\overset{i}{\sum}}\, a_{i}^{\prime}\, t^{2j+1}\ +\ \underset{j=1}{\overset{i}{\sum}}\, a_{i}^{\prime \prime}\, t^{-2j+1}\right)\ =\\
&&&&\\
&&& = & \underset{A}{\left(\underbrace{\underset{i=0}{\overset{\frac{n-1}{2}}{\sum}}\, \underset{j=0}{\overset{i}{\sum}}\, b_{i}\, t^{2j+1}}\right)}\ +\ \underset{B}{\left(\underbrace{\underset{i=0}{\overset{\frac{n-1}{2}}{\sum}}\, \underset{j=1}{\overset{i}{\sum}}\, b_{i}^{\prime}\, t^{-2j+1}}\right)},\\
\end{array}
\]
\noindent where $a_i, a_{i}^{\prime}, a_{i}^{\prime \prime}, b_i$ and $b_i^{\prime}$ are coefficients for all $i$. \smallbreak

\smallbreak
We now deal with terms in $B$ that consist of elements in ${B_{{\rm ST}}}^{aug}\, \backslash\, B_{{\rm ST}}$, since all terms in $A$ are in $B_{{\rm ST}}$. The highest exponent in the terms in $B$ is $(-n+3)>-n$ and thus, we may continue applying Lemmata~\ref{fl1} and \ref{fl2} until we eliminate all negative exponents. All resulting monomials will eventually belong to $B_{{\rm ST}}$ and the proof is concluded.
\end{proof}

\begin{remark}\label{remsp}\rm
It is possible to find a closed formula for $t^{-n}$, $n\in \mathbb{N}$ in KBSM(ST). More precisely, simple calculations lead to the following result:
\[
t^{-n}\ \widehat{\underset{{\rm skein}}{\cong}}\ \begin{cases} \underset{i=0}{\overset{n/2}{\sum}}\, a_i\, t^{2i}&,\ {\rm for}\ n\ {\rm even}\\
\underset{i=0}{\overset{\frac{n-1}{2}}{\sum}}\, a^{\prime}_i\, t^{2i+1}&,\ {\rm for}\ n\ {\rm odd}\\ \end{cases},
\]
\noindent where $a_i, a^{\prime}_i$ coefficients, for all $i$.
\end{remark}

\subsection{The new basis of KBSM(ST)}

We are now in position to prove the main result of this section. More precisely:

\begin{thm}\label{mr}
The set $B_{\rm ST}$ forms a basis for the Kauffman bracket skein module of the solid torus.
\end{thm}

\begin{proof}
Consider $\tau^{\prime}_{0, n} \in B^{\prime}_{{\rm ST}}$ of index $n$. We apply Corollary~\ref{cor1} and obtain a sum of elements in ${B_{{\rm ST}}}^{aug}$ with $(n+1)$ being the highest exponent appearing in the sum. More precisely, for $n$ odd we have:

\[
\begin{array}{lclcl}
B^{\prime}_{{\rm ST}}\, \ni\, \tau^{\prime}_{0, n} & \widehat{\underset{{\rm skein}}{\cong}} & \underset{\in B_{\rm ST},\, \forall i}{\underbrace{\underset{i=0}{\overset{\frac{n+1}{2}}{\sum}}\, a_i\, t^{2i}}} & + & \underset{\in {B_{\rm ST}}^{aug}}{\underbrace{\underset{i=1}{\overset{\frac{n-1}{2}}{\sum}}\, \underline{a_i\, t^{-2i}}}}\\
&&&&\\
& \widehat{\underset{{\rm skein}}{\cong}} & \underset{\in B_{\rm ST},\, \forall i}{\underbrace{\underset{i=0}{\overset{\frac{n+1}{2}}{\sum}}\, a_i\, t^{2i}}} & + & \underset{\in B^{\prime}_{\rm ST}}{\underbrace{\underset{i=1}{\overset{\frac{n-1}{2}}{\sum}}\, \left(\underset{j=0}{\overset{i}{\sum}}\,  a_i^{\prime}\,  \tau^{\prime}_{0, 2j}\right)}}
\end{array}
\]


\noindent The monomial of highest order in the sum $\underset{i=1}{\overset{\frac{n-1}{2}}{\sum}}\, \left(\underset{j=0}{\overset{i}{\sum}}\,  a_i^{\prime}\,  \tau^{\prime}_{0, 2j}\right)$ is $\tau_{0, n-1}^{\prime}$, which according to the ordering relation of Definition~\ref{order}, is of less order than the initial monomial $\tau_{0, n}^{\prime}$. Hence, if we continue applying the same procedure on the resulting monomials, we will eventually reach at a sum of elements in $B_{{\rm ST}}$, with $(n+1)$ being the highest exponent.
\smallbreak
Thus, we have proved that any monomial $\tau^{\prime}$ in $B^{\prime}_{{\rm ST}}\, :\, ind(\tau^{\prime})=n$, can be written as a sum of elements in $B_{{\rm ST}}$ and in particular, $\tau^{\prime}\, \widehat{\underset{{\rm skein}}{\cong}}\, \underset{i=0}{\overset{n+1}{\sum}}\, a_i\, t^i$, for some coefficients $a_i$. Moreover, $t^{n+1}$ is the monomial with the highest order appearing when converting the monomial $\tau^{\prime} \in B^{\prime}_{{\rm ST}}$ of index $n$ to elements in $B_{{\rm ST}}$. Hence, the sets $B^{\prime}_{{\rm ST}}$ and $B_{{\rm ST}}$ are related via an infinite lower triangular matrix with invertible elements in the diagonal. Thus, the set $B_{{\rm ST}}$ forms a basis of KBSM(ST).
\end{proof}

\section{A basis for the Kauffman bracket skein module of $L(p,1)$ via braids}\label{kblens}

In this section we solve the infinite system of Equations~(\ref{eqbbm}) for $q=1$, which is equivalent to computing the Kauffman bracket skein module of the lens spaces $L(p,1)$. Recall that this infinite system of equations is obtained from elements in the basis of KBSM(ST), $B_{\rm ST}$, by performing braid band moves and then imposing to the generic invariant $V$ for knots and links in ST relations of the form $V_{\widehat{t^n}}\, =\, V_{\widehat{bbm(t^n)}}$, for all $n\in \mathbb{N}$, where $bbm(t^n)\, =\, t^p\, t_1^n\, \sigma_1^{\pm 1}$ (recall Theorem~\ref{markov}). Recall also that the unknowns in the system are the $s_i$'s, coming from the fourth rule of the trace function in Theorem~\ref{tr}, that is, $tr(t^n)\, =\, s_n$, for all $n\in \mathbb{Z}$.

\subsection{The infinite system}\label{infs}

In this subsection we present a series of results toward the solution of the infinite system. Our main task is the evaluation of $tr(t^p\, t_1^n\, \sigma_1^{\pm 1})$ for $n\in \mathbb{N}$. Recall now that the trace function is defined using the $t_i^{\prime}$'s and not the $t_i$'s that we obtain when performing a bbm on elements in $B_{\rm ST}$ (recall the fourth rule of Theorem~\ref{tr}). Thus, our first task is to express the monomials $t^p\, t_1^n\, \sigma_1^{\pm 1}\in \Sigma$ to sums of monomials in $\Sigma^{\prime}$ (recall Eq.~(\ref{bHeck})). In \cite{DL2}, a method for expressing monomials in $\Sigma^{\prime}$ to sums of monomials in $\Sigma$ (and vice-versa) is presented on the generalized Hecke algebra of type B, $H_{1, n}$, level. More precisely, it is shown that a monomial $\tau \in \Sigma$ can be written as a sum of elements in $\Sigma^{\prime}$, such that the homologous word of $\tau$ in $\Sigma^{\prime}$, i.e. the monomial $\tau$ where all $t_i$'s are replaced by $t_i^{\prime}$'s, is the monomial of the highest order appearing in the resulting sum. All other monomials appearing in the sum have smaller index, or equal index but different exponents (recall the ordering relation of Definition~\ref{order}). Finally, the braiding ``tails'', i.e. the $\sigma_i$'s that come after the $t_i$'s in monomials in $\Sigma$, are eliminated in this process.

\smallbreak

Since now the generalized Temperley-Lieb algebra of type B, $TL_{1, n}$, is a quotient of $H_{1, n}$ over the ideal generated by elements in Eq.~(\ref{ideal}) that only involves the $\sigma$'s, and since the only braiding generator in the monomials $bbm(t^n)$ is $\sigma_1$, it follows that the results presented in \cite{DL2} are also true in $TL_{1, n}$ for the monomials $bbm(t^n)=t^p\, t_1^n\, \sigma_1^{\pm 1}$. Hence, we obtain the following result:

\begin{lemma}\label{l1}
For $n, p \in \mathbb{N}\backslash \{0\}$, the following holds in $TL_{1, n}$:
\[
t^p\, t_1^n\, \sigma_1\ \widehat{\cong}\ \underset{i=0}{\overset{n}{\sum}}\, a_i\, t^{p+i}\, {t_1^{\prime}}^{n-i},
\]
\noindent where $a_i$ coefficients for all $i$.
\end{lemma}

Note that in the monomial of highest order in $\underset{i=0}{\overset{n}{\sum}}\, a_i\, t^{p+i}\, {t_1^{\prime}}^{n-i}$ is the homologous monomial of $t^p\, t_1^n$, $t^p\, {t^{\prime}_1}^n$. In order now to evaluate $tr\left(\underset{i=0}{\overset{n}{\sum}}\, a_i\, t^{p+i}\, {t_1^{\prime}}^{n-i}\right)$, we first perform Lemma~\ref{lemst} on the monomials in the sum and we have the following result:

\begin{lemma}\label{l2}
For $n, p \in \mathbb{N}\backslash \{0\}$, the following holds in $TL_{1, n}$:
\[
t^p\, {t_1^{\prime}}^n \ \widehat{\underset{{\rm skein}}{\cong}}\ a_i \, t^{p-n}\, +\, b_i\, t^{p+n},
\]
\noindent where $a_i, b_i$ are coefficients for all $i$.
\end{lemma}

Hence, we have that $tr(t^p\, t_1^n\, \sigma_1^{\pm 1})\ =\ \underset{i=0}{\overset{n}{\sum}}\, \left(a_i^{\prime}\, s_{p+n}\, +\, b_i^{\prime}\, s_{p-n+2i}\right)$, for some coefficients $a_i^{\prime}$ and $b_i^{\prime}$, since:
\[
tr(t^p\, t_1^n\, \sigma_1^{\pm 1})\ =\ tr\left(\underset{i=0}{\overset{n}{\sum}}\, a_i\, t^{p+i}\, {t_1^{\prime}}^{n-i}\right)\ =\ tr\left(\underset{i=0}{\overset{n}{\sum}}\, \left(a_i^{\prime}\, t^{p+n}\, +\, b_i^{\prime}\, t^{p-n+2i}\right)\right).
\]

\begin{remark}\rm
Consider now a negative braid band move applied on elements in $B_{\rm ST}$, i.e. $t^n\, \rightarrow\, t^p\, t_1^n\, \sigma_1^{-1}$. Observe  that $t^p\, t_1^n\, \sigma_1^{-1}\, =\, t^p\, t_1^{n-1}\, \sigma_1\, t\, \widehat{\cong}\ t^{p+1}\, t_1^{n-1}\, \sigma_1$. Thus, applying Lemmata~\ref{l1} and \ref{l2} on $t^{p+1}\, t_1^{n-1}\, \sigma_1$ we obtain:
\[
t^{p+1}\, t_1^{n-1}\, \sigma_1\ \widehat{\cong}\ \underset{i=0}{\overset{n-1}{\sum}}\, a_i\, t^{p+1+i}\, {t_1^{\prime}}^{n-1-i}\ \widehat{\underset{{\rm skein}}{\cong}}\ \underset{i=0}{\overset{n-1}{\sum}}\, \left(b_i \, t^{p+n}\, +\, c_i\, t^{p-n+2+2i}\right),
\]
\noindent for some coefficients $a_i, b_i$ and $c_i$. Evaluating now the trace of these elements, we have that
\[
tr(t^p\, t_1^n\, \sigma_1^{-1})\ =\ \underset{i=0}{\overset{n-1}{\sum}}\, \left(b_i \, s_{p+n}\, +\, c_i\, s_{p-n+2+2i}\right).
\]
Less number of unknowns of the infinite system appear in the equation $V_{\widehat{t^n}}\ =\ V_{\widehat{bbm_-(t^n)}}$, compared to the equation $V_{\widehat{t^n}}\ =\ V_{\widehat{bbm_+(t^n)}}$, and thus, we may consider only the positive braid band moves for obtaining the equations of the infinite system for $n\in \mathbb{N}$.
\end{remark}

\subsection{Dealing with indices greater than $p$}\label{modp}

We are now in position to prove the first fundamental result toward the solution of the infinite system. We have the following:

\begin{prop}\label{th1}
For $n\in \mathbb{N}$ he following relations hold in KBSM(L(p,1)) for $n\in \mathbb{N}$:
\[
s_{p+n}\ =\ \begin{cases} \underset{i=0}{\overset{n/2}{\sum}}\, (c_i\, s_{2i}\, +\, d_i\, s_{p-2i}) &,\, {\rm for}\ n\ {\rm even}\\
\underset{i=0}{\overset{(n-1)/2}{\sum}}\, (c^{\prime}_i\, s_{2i-1}\, +\, d^{\prime}_i\, s_{p-2i-1}) &,\,{\rm for}\ n\ {\rm odd} \end{cases},
\]
\noindent where $c_i, d_i, c^{\prime}_i, d^{\prime}_i$ are coefficients.
\end{prop}

\begin{proof}
The following diagram describes how the equations of the infinite system are obtained:
\[
\begin{array}{ccccccc}
t^n & \overset{bbm}{\longrightarrow} & t^pt_1^n\, \sigma_1 & \hat{\cong} & \underset{i=0}{\overset{n}{\sum}}\, a_i\, t^{p+i}\, {t_1^{\prime}}^{p-i} & \hat{\underset{{\rm skein}}{\cong}} & \underset{i=0}{\overset{n}{\sum}}\, \left(t^{p-n+2i}\, +\, t^{p+n}\right) \\
|   &                 &                     &                 &                                                                          &                                        &       |\\
tr   &                 &                     &                 &                                                                          &                                        &       tr\\
\downarrow   &                 &                     &                 &                                                                          &                                        &       \downarrow\\
s_n &                 &                    &                  &                                                                           &                                                                                                                  &        \underset{i=0}{\overset{n}{\sum}}\, s_{p-n+2i}          \\      
|   &                 &                     &                 &                                                                          &                                        &       |\\
V   &                 &                     &                 &                                                                          &                                        &       V\\
\downarrow   &                 &                     &                 &                                                                          &                                        &       \downarrow\\   
s_n &                 &                    &                  &            =                                                               &                                                                                                                  &        \left(-\, \frac{1+u^2}{u} \right)\, u^{2e}\cdot \underset{i=0}{\overset{n}{\sum}}\, s_{p-n+2i}          \\
\end{array}
\]

\noindent Hence, $s_n\, =\, \left(-\, \frac{1+u^2}{u} \right)\, u^{2e}\cdot \underset{i=0}{\overset{n}{\sum}}\, s_{p-n+2i}$, for all $n\in \mathbb{N}$. 

\bigbreak

We now consider the following cases: 

\smallbreak

\begin{itemize}
\item For $n=2k,\, k\in \mathbb{N}$, we have:

\[
\begin{array}{lclcccl}
{\rm For}\ n=0 & : & s_0\, =\, s_{p} & \Leftrightarrow & s_p & = & s_0\, :=\, 1\\
&&&&&&\\
{\rm For}\ n=2 & : & s_2\, =\, s_{p-2}\, +\, s_p\, +\, s_{p+2} & \Leftrightarrow & s_{p+2} & = & s_0\, +\, s_2\, +\, s_{p-2}\\
&&&&&&\\
{\rm For}\ n=4 & : & s_4\, =\, s_{p-4}\, +\, s_{p-2}\, +\, s_p\, +\, s_{p+2}\, +\, s_{p+4} & \Leftrightarrow & s_{p+4}& =& s_0\, +\, s_2\, +\, s_4\, +\, \\
&&&&&  & +\, s_{p-4}\, +\, s_{p-2} \\
\vdots &&&&&&\\
{\rm For}\ n=2k & : &  & \Leftrightarrow & s_{p+2k} & = & \underset{i=0}{\overset{k}{\sum}}\, s_{2i}\, +\, \underset{i=1}{\overset{k}{\sum}}\, s_{p-2i}\\
\end{array}
\]

\bigbreak

\item For $n=2k+1,\, k\in \mathbb{N}$, we have:

\[
\begin{array}{cclcccl}
{\rm For}\ n=1 & : & s_1\, =\, s_{p-1}\, +\, s_{p+1} & \Leftrightarrow & s_{p+1} & = & s_1\, +\, s_{p-1}\\
&&&&&&\\
{\rm For}\ n=3 & : & s_3\, =\, s_{p-3}\, +\, s_{p-1}\, +\, s_{p+1}\, +\, s_{p+3} & \Leftrightarrow & s_{p+3}& =& s_1\, +\, s_3\, +\, s_{p-3}\, +\, \\
&&&&&  & +\, s_{p-1} \\
\vdots &&&&&&\\
{\rm For}\ n=2k+1 & : &  & \Leftrightarrow & s_{p+2k+1} & = & \underset{i=0}{\overset{k}{\sum}}\, (s_{2i+1} + s_{p-2i-1})\\
\end{array}
\]
\end{itemize}
The proof is concluded.
\end{proof}

An immediate result of Proposition~\ref{th1} is that every unknown of the infinite system with index greater than or equal to $p$, can be written as a sum of the unknowns with index less than $p$. That is:

\begin{cor}\label{cor3}
For $k\in \mathbb{N}$, the following hold in KBSM($L(p,1)$):
\[
s_{p+k}\ =\ \underset{i<p}{\sum}\, q_i\, s_i, \quad {\rm where}\ i\in \mathbb{Z}\ {\rm and}\ q_i\ {\rm coefficients}.
\]
\end{cor}

\begin{remark}\rm
Note that since $tr(t^n)\, =\, s_n$, we have proved that the loop elements $t^{p+i}\in {\rm KBSM(ST)},\, i\in \mathbb{N}$ can be expressed as a sum of elements in the form $t^{p-j},\, j\in \mathbb{N}$. Roughly speaking, the exponents of the $t$'s are mod$p$ in KBSM($L(p,1)$).
\end{remark}

It is important to observe now that depending on $p$, the exponents on some monomials of the resulting sum may be negative. We shall deal with negative indices in the next subsection.

\subsection{Dealing with negative indices}\label{nelens}

We now deal with the unknowns of the infinite system that have negative indices and express them in terms of $s_i$'s where $i\in \mathbb{N}$. Recall that $s_{-n} = tr(t^{-n})$ and that $t^{-n}\, \widehat{\underset{{\rm skein}}{\cong}}\, \underset{i=0}{\overset{n}{\sum}}\, a_i\, t^i$ (Proposition~\ref{pr1} ). Thus, we have that:
\[s_{-n}\, =\, tr(t^{-n})\, =\, tr\left(\underset{i=0}{\overset{n}{\sum}}\, a_i\, t^i \right)\, =\, \underset{i=0}{\overset{n}{\sum}}\, a_i\, s_i.
\]

If instead of Proposition~\ref{pr1} we use relations from Remark~\ref{remsp}, we obtain the following:

\begin{lemma}\label{nein}
The following relations hold in KBSM($L(p,1)$) for all $n\in \mathbb{N}$:
\[
s_{-n}\ = \begin{cases} \underset{i=0}{\overset{n/2}{\sum}}\, s_{2i} &, {\rm for}\ n\ {even}\\
\underset{i=0}{\overset{(n-1)/2}{\sum}}\, s_{2i+1} &, {\rm for}\ n\ {odd}
\end{cases}
\]
\end{lemma}

\begin{proof}
We have that:
\[
\begin{array}{cccc}
t^{-n} & \widehat{\underset{{\rm skein}}{\cong}} & \underset{i=0}{\overset{n/2}{\sum}}\, a_i\, t^{2i},\, {\rm for}\, n\ {\rm even\ \ \ or} & \underset{i=0}{\overset{(k-1)/2}{\sum}}\, b_i\, t^{2i+1},\, {\rm for}\, k\ {\rm odd}\\
| & & | & |\\
{\rm V} & & {\rm V} & {\rm V}\\
\downarrow & & \downarrow & \downarrow \\
s_{-n} & & \underset{i=0}{\overset{n/2}{\sum}}\, a_i\, s_{2i} & \underset{i=0}{\overset{(n-1)/2}{\sum}}\, b_i\, s_{2i+1} \\
\end{array}
\]

\noindent for some coefficients $a_i, b_i$. Omitting the coefficients we have that:
\[
s_{-2}\, =\, s_{0}\, +\, s_2,\ s_{-4}\, =\, s_0\, +\, s_2\, +\, s_4,\ \ldots, \ s_{-j}\, =\, s_0\, +\, s_2\, +\, \ldots\, +\, s_{j},\ {\rm for}\ j\ {\rm even\ and} 
\]
\[
s_{-1}\, =\, s_{1},\ s_{-3}\, =\, s_1\, +\, s_3,\ \ldots, \ s_{-j}\, =\, s_1\, +\, s_3\, +\, \ldots\, +\, s_{j},\ {\rm for}\ j\ {\rm odd}.
\]
\noindent The result follows.
\end{proof}

\subsection{The solution of the infinite system}

We are now in position to present the solution of the infinite system that is equivalent to computing the Kauffman bracket skein module of the lens spaces $L(p,1)$. We first generalize Proposition~\ref{th1} and in particular we show that the unknowns with indices greater than or equal to $p$ can be written as sums of unknowns with non-negative indices and also that these indices are $\mod p$. Indeed, we have the following:

\begin{prop}\label{fp1}
The following relations hold in KBSM($L(p,1)$) for all $n\in \mathbb{N}$:
\[
s_{p+n}\ =\ \underset{i=0}{\overset{p}{\sum}}\, a_i\, s_{i}.
\]
\end{prop}

\begin{proof}
For $n\leq p$ the results follows from Proposition~\ref{th1}. Let now $n>p$ and consider the case where $n$ is even. The case $n$ being odd follows similarly.
\smallbreak
Let $j\in \mathbb{N}\backslash \{0\}$ such that $n=p+j$. Omitting the coefficients, we have:
\[
\begin{array}{lcl}
s_{p+n} & \overset{Prop.~\ref{th1}}{=} & \underset{i=0}{\overset{n/2}{\sum}}\, (s_{2i}\, +\, s_{p-2i})\ = \\
&&\\
& = & \underset{A}{\left[\underbrace{(s_0\, +\, s_p)\ +\ (s_{2}\, +\, s_{p-2})\ +\ \ldots\ +\ (s_p\, +\, s_0)}\right]} \ +\\
&&\\
& = & \underset{B}{\left[\underbrace{(s_{p+2}\, +\, s_{-2})\ +\ (s_{p+4}\, +\, s_{-4})\ +\ \ldots\ +\ (s_{p+j}\, +\, s_{-j})}\right]} \ +\\
\end{array}
\]
\noindent Note now that all indices in terms in $A$ are in $\{0, 1, \ldots, p\}$ and thus, we focus on terms in $B$, i.e. terms of the form $s_{-k}$ and $s_{p+k}$ for $k\in \{1, 2, \ldots, j \}$.

\smallbreak

We distinguish the following cases:
\begin{itemize}
\item[{\rm Case\ I:}] We deal first with terms of the form $s_{-k}$ for $k\in \{1, 2, \ldots, j \}$. We apply Lemma~\ref{nein} on these terms and we obtain sums of unknowns with all indices being less than or equal to $j$. Thus, if $j<p$ the result follows. If $j>p$, we repeat the same procedure until we reduce all indices $\mod p$.

\smallbreak

\item[{\rm Case\ II:}] We now deal with terms of the form $s_{p+k}$ for $k\in \{1, 2, \ldots, j \}$. We apply Proposition~\ref{th1} and obtain elements in $s_i$'s such that $i$ is either in $\mathbb{Z}\backslash \mathbb{N}$ (see Case I), or in $\mathbb{N}$ such that $i<p$. Continuing that way, we eventually obtain a sum of elements of the form $\underset{i=0}{\overset{p}{\sum}}\, s_i$.
\end{itemize}
\end{proof}

We shall now reduce the indices in the unknowns further. Let 
\[
I=\{0, 1, \ldots, \lfloor p/2 \rfloor\},
\]
\noindent where $\lfloor k \rfloor$ denotes the greatest integer less than or equal to $k$. We have:

\begin{thm}\label{mth1}
For all $n\in \mathbb{N}\backslash I$ we have that $s_n\, =\, \underset{i\in I}{\sum}\, s_i$.
\end{thm}

\begin{proof}
From Proposition~\ref{fp1}, it follows that it suffices to show that $s_n\, =\, \underset{i\in I}{\sum}\, s_i$ for $n\in \{\lfloor p/2 \rfloor +1, \ldots, p \}$. 

\smallbreak

Consider elements in ${B_{\rm ST}}^{aug} \backslash B_{\rm ST}$ of the form $t^{-n}$ for $n\in I$. We perform bbm's on these elements in order to obtain equations for the infinite system. Using Lemma~\ref{nein} and Proposition~\ref{th1}, we may convert the unknowns with negative indices to sums of unknowns with indices in $I$ and thus, we conclude that the unknowns with indices $I$ suffice to generate all unknowns $s_k$, where $k\notin I$. Equivalently, the elements $t^k$, where $k\notin I$, can be expressed as sums of elements of the form $t^n$, where $n\in I$. More precisely, we have:

\[
\begin{array}{clclcl}
\bullet & t^0\ \overset{bbm}{\rightarrow}\ t^p\, \sigma_1^{\pm 1} & \overset{V}{\Rightarrow} & 1:=s_0\, =\, s_p & \Rightarrow & s_p\ =\ s_0\\
&&&&&\\
\bullet & t^{-1}\ \overset{bbm}{\rightarrow}\ t^p\, t_1^{-1}\sigma_1^{\pm 1} & \overset{V}{\Rightarrow} & s_{-1}\, =\, s_{p-1} & \overset{Lemma~\ref{nein}}{\underset{s_{-1}=s_1}{\Rightarrow}}& s_{p-1}\ =\ s_1\\
&&&&&\\
\bullet & t^{-2}\ \overset{bbm}{\rightarrow}\ t^p\, t_1^{-2} \sigma_1^{\pm 1} & \overset{V}{\Rightarrow} & s_{-2}\, =\, s_{p-2}+s_p+s_{p+2} & \overset{Lemma~\ref{nein}\, \& \, Prop.~\ref{fp1}}{\underset{\left\{\begin{matrix} s_{-2}& = & s_0+s_2\\ s_{p+2} & = & \underset{i=0}{\overset{p}{\sum}}\, s_i\end{matrix}\right\}}{\Rightarrow}}& s_{p-2}\ =\ \underset{i=0}{\overset{p-3}{\sum}}\, s_i\\
&&&&&\\
\vdots &&\vdots& &\vdots& \\
\end{array}
\]

\[
\begin{array}{llcl}
\bullet & t^{-\lfloor p/2 \rfloor +1}\ \overset{bbm}{\rightarrow}\ t^p\, t_1^{-\lfloor p/2 \rfloor +1} \sigma_1^{\pm 1} & \overset{V}{\Rightarrow} & s_{-\lfloor p/2 \rfloor +1}\, =\, s_{p-\lfloor p/2 \rfloor -1}+\ldots+s_{p+\lfloor p/2 \rfloor+1}\\
&&&\\
&& \Rightarrow & s_{p-\lfloor p/2 \rfloor +1}\ =\ \underset{i=0}{\overset{p-\lfloor p/2 \rfloor}{\sum}}\, s_i
\end{array}
\]

Note that in all computations above, we omit the coefficients in order to simplify the algebraic expressions. The proof is now concluded.
\end{proof}

\begin{cor}
The set $\mathcal{B}_p\ := \{t^n,\, n\in I \}$ spans the Kauffman bracket skein module of the lens spaces $L(p,1)$.
\end{cor}

We now show that the set $\mathcal{B}_p$ is linearly independent and thus, we prove the main theorem of the paper, Theorem~\ref{important}:

\begin{proof}
From the discussion and the results presented above, it follows that it suffices to show that $t^{\lfloor p/2 \rfloor}$ cannot be written as a sum of elements in $B_p \backslash \{t^{\lfloor p/2 \rfloor}\}$, or equivalently, that the unknown $s_{\lfloor p/2 \rfloor}$ cannot be written as a sum of unknowns of indices in $I\backslash \{\lfloor p/2 \rfloor\}$. We consider the most interesting case that $p$ is even. The case that $p$ is odd follows immediately from the proof of Theorem~\ref{mth1}.

\smallbreak

Note that since $t^{p/2}$ is in the basis of KBSM(ST), and since KBSM($L(p,1)$)$=$KBSM(ST)$/<{\rm band\ moves}>$, it suffices to show that the braud band moves do not affect $t^{\lfloor p/2 \rfloor}$. Recall first that isotopy in $L(p,1)$ can be viewed as isotopy in ST together with one of the two different types of band moves (Theorem~\ref{markov}). So far we have been using the $\alpha$-type band moves that are immediately translated on the level of braids by the braid band moves. In order to prove that $t^{p/2}$ for $p$ even, cannot be expressed as a sum of looping generators of exponents in $I\backslash \lfloor p/2 \rfloor$, we use the $\beta$-type band moves. This simplifies the calculations involved a lot.

\smallbreak

Indeed, consider $t^{p/2} \in B_{\rm ST}$ and perform the $\beta$-type band move as illustrated in Figure~\ref{betabm}. Note that the result of the performance of a type $\beta$ band move on a looping generator is not a mixed braid. But as shown in Figure~\ref{betabm}, isotopy in ST results in the looping generator $\widehat{t^{p-p/2}}\, =\, \widehat{t^{p/2}}$, that we braid and obtain $t^{p/2}$ again. Hence, $t^{p/2}$ cannot be written as sum of elements in $\mathcal{B}_p\backslash \{t^{p/2}\}$ and the proof is now concluded.
\end{proof}

\begin{remark}\rm
It is possible to prove that $t^{p/2}$ cannot be written as a sum of elements in $\mathcal{B}_p\backslash \{t^{p/2}\}$ using the braid band moves, i.e. the $\alpha$-type band moves. The difficulty lies in keeping track of all coefficients involved, since following the same procedure as in the proof of Theorem~\ref{mth1} for $t^{-p/2}\, \overset{bbm}{\rightarrow}\, t^p\, t_1^{-p/2}\, \sigma_1^{\pm 1}$ (applying Lemma~\ref{nein} and Proposition~\ref{th1}), results in an equation of the form $\underset{i=0}{\overset{p/2}{\sum}}\, a_i\, s_i\ =\ \underset{i=0}{\overset{p/2}{\sum}}\, b_i\, s_i$, where $a_i, b_i$ coefficients for all $i$. We would then have to prove that $a_i = b_i$ for all $i$, which is very complicated and technical.
\end{remark}

\begin{remark}\rm
The case $L(0,1)\, \cong \, S^1 \times S^2$ is of special interest and in \cite{D10}, the algebraic approach is applied in order to compute KBSM($S^1 \times S^2$). As shown in \cite{HP1}, KBSM($S^1 \times S^2$) contains torsion, and in \cite{D10} we show how torsion is detected by solving the corresponding infinite system of equations (\ref{eqbbm}). This emphasizes on the importance of the algebraic approach based on unoriented braids that we present in this paper.
\end{remark}


\section{The diagrammatic approach for computing KBSM($L(p,q)$)}\label{diag}

In this section we compute the Kauffman bracket skein module of the lens spaces $L(p,q)$ using a diagrammatic method that is different to the diagrammatic method used in \cite{HP}. We first demonstrate this method for the $q=1$ case and we then extend it for $q>1$. It is worth mentioning that in \cite{D4}, a new basis for the Kauffman bracket skein module of the complement of $(2,2p+1)$-torus knots is presented via this new method, which is based on unoriented braids.

\subsection{The $L(p,1)$ case}\label{lcase1}

In this subsection we describe how the basis $\mathcal{B}_p$ can be obtained using a diagrammatic approach based on unoriented braids. Instead of the $\alpha$-band move, we will be using the $\beta$-type band move throughout this section (recall Figure~\ref{bmov}).

\smallbreak

Let $t^n\in B_{\rm ST}$. Apply a $\beta$-type band move as shown in the first two illustrations of Figure~\ref{betabm}. The red strand corresponds to the strand that appears after the band move is performed, which wraps around the surgery curve $I$, $p$ times. The orange strands in the middle illustration indicate the places where isotopy in ST is performed in order to unwrap the moving strand from the surgery strand.

\begin{figure}
\begin{center}
\includegraphics[width=5.5in]{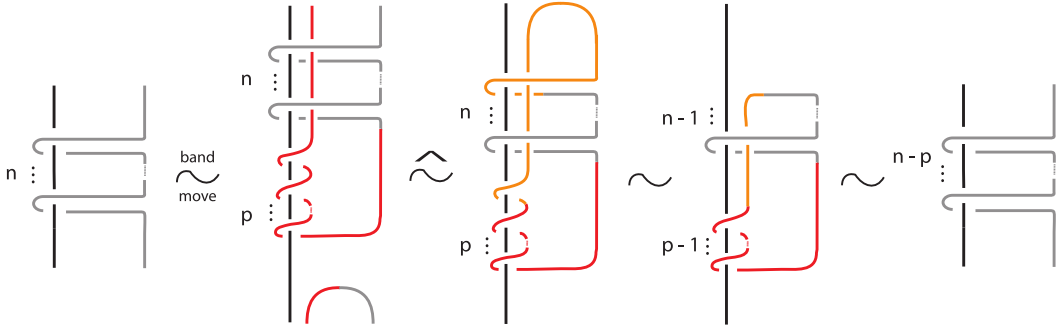}
\end{center}
\caption{The $\beta$-type band move on an element in $B_{\rm ST}$ and its result in KBSM($L(p,1)$).}
\label{betabm}
\end{figure}

\bigbreak

As shown in Figure~\ref{betabm}, using the $\beta$=band move and isotopy in ST, we may reduce $t^{n}$ to $t^{n-p}$ for all $n\in \mathbb{Z}$. In particular, for $n\in \mathbb{N}$, we have that:

\[
\begin{array}{ccccccc}
n=0 & : & t^0\ =\ t^p & , & n=1 & : & t^1\ =\ t^{p-1}\\
    &   &            &   &     &   &\\
n=2 & : & t^2\ =\ t^{p-2} & , & n=3 & : & t^3\ =\ t^{p-3}\\
 \vdots   &   &    \vdots        &  & \vdots      & &\vdots  \\
n=\lfloor p/2 \rfloor -1 & : & t^{\lfloor p/2 \rfloor -1}\ =\ t^{p- \lfloor p/2 \rfloor + 1} &  & & &
\end{array}
\] 

\noindent and we conclude that the $t^n$'s for $n\in \{\lfloor p/2 \rfloor +1, \ldots, p\}$ are not in the basis of KBSM($L(p,1)$).

\smallbreak

Similarly, for $n\in \mathbb{Z}\backslash \mathbb{N}$, we have:

\[
\begin{array}{ccccccc}
n=-1 & : & t^{-1}\ =\ t^{p+1} & , & n=-2 & : & t^{-2}\ =\ t^{p+2}\\
    &   &            &   &     &   &\\
n=-3 & : & t^{-3}\ =\ t^{p+3} & , & n=-4 & : & t^{-4}\ =\ s_{p+4}\\
 \vdots   &   &    \vdots        &  & \vdots      & &\vdots  \\
\end{array}
\] 

\noindent and we conclude that $t^{p+k}\, \sim\, t^{-k}$, for $k\in \mathbb{N}$. Recall now that $t^{-k}\, \widehat{\underset{{\rm skein}}{\cong}}\, \underset{i=0}{\overset{k}{\sum}}\, t^i$ (Proposition~\ref{pr1}). Thus, $t^{p+1}\, \sim\, t^{-1}\, =\ t$, etc. We conclude again that the $t^n$'s for $n>p$ are not in the basis of KBSM($L(p,1)$). Equivalently, we have that the set $\mathcal{B}_p$ spans KBSM($L(p,1)$). Linear independence follows then from the proof of Theorem~\ref{important}.

\subsection{The $L(p,q),\, q>1$ cases}\label{lpq}

In this subsection we adjust the proof of Theorem~\ref{important} for the case of $L(p, q),\, q>1$ and in particular we show that the diagrammatic approach using the unoriented braids that we presented in \S~\ref{lcase1}, also works for $q>1$. Hence, we shall conclude that the set $\mathcal{B}_p$ forms a basis for KBSM($L(p,q)$). 

\smallbreak

We first note that the band moves are more complicated for $q>1$, in the sense that when a moving strand of a mixed link approaches the surgery curve $\hat{I}$, $q$ new strands will appear wrapping around $\hat{I}$, $p$-times. For an illustration see Figure~\ref{bandmm}, for $p=3$ and $q=2$ and where band move is shortened to bm.

\begin{figure}[H]
\begin{center}
\includegraphics[width=3.5in]{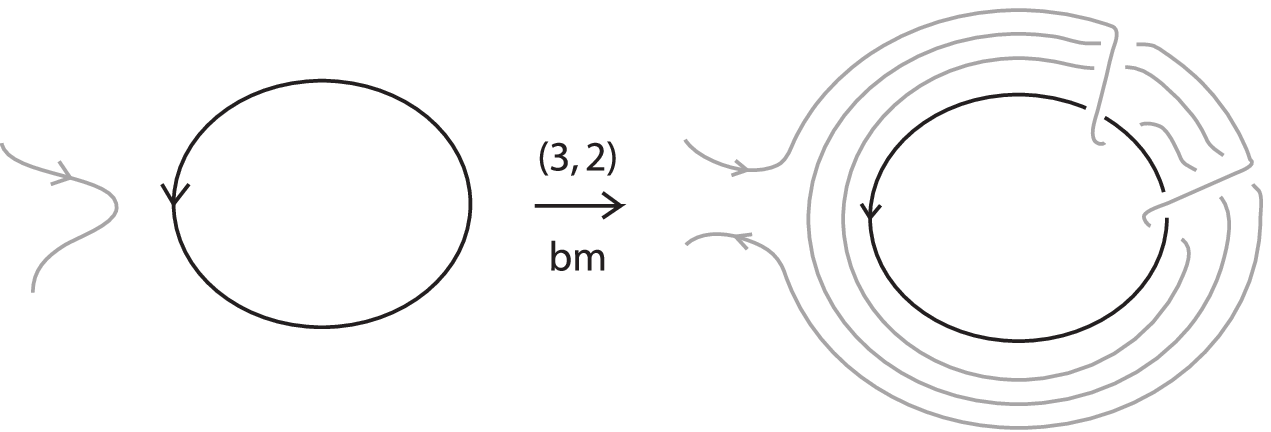}
\end{center}
\caption{The $\beta$-type band move for $L(3, 2)$.}
\label{bandmm}
\end{figure}

On the unoriented braid level, a generic $(p, q)$-band move is illustrated in Figure~\ref{bbm1}. The shaded area depicts the braided $(p, q)$-torus knots that appears after the performance of the band move, coming from the fact that this torus knot bounds a disc in $L(p, q)$. Note also that the standard closure for mixed braids on the resulted ``braid'' is a $\beta$-band move on the mixed link level.

\begin{figure}[H]
\begin{center}
\includegraphics[width=2.5in]{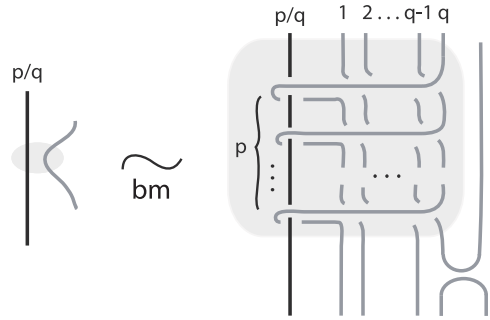}
\end{center}
\caption{The band move for $L(p,q), q>1$.}
\label{bbm1}
\end{figure}

Our starting point is the basis $B_{\rm ST}$ of the Kauffman bracket skein module of the solid torus. The idea is the same as in the $q=1$ case. More precisely, we consider elements $t^n\in\, B_{\rm ST}$ and perform the $\beta$-type band moves. Then, using isotopy in ST we may reduce $t^{n}$ to $t^{n-p}$ for all $n$. In Figure~\ref{bbmn} we demonstrate this process for $L(3, 2)$ and in particular we show how to reduce $t^2$ to $t$ in $L(3,2)$.

\begin{figure}
\begin{center}
\includegraphics[width=4.7in]{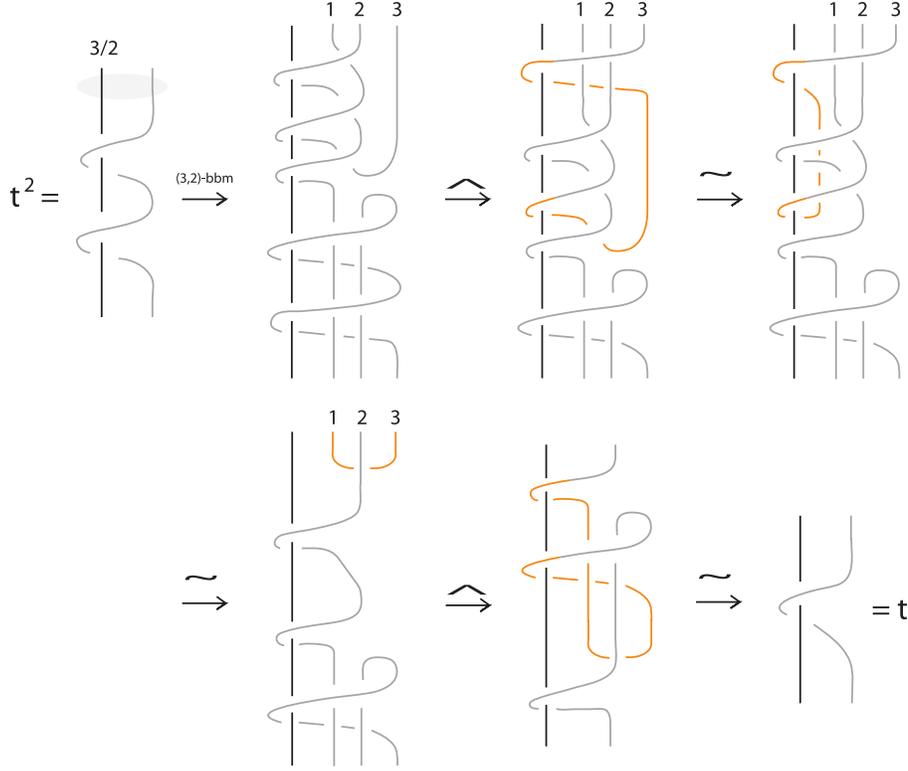}
\end{center}
\caption{ $t^2\, \sim t$ in $L(3, 2)$.}
\label{bbmn}
\end{figure}

Thus, for $k\notin I$, we have that $t^{k}\, \widehat{\underset{\rm skein}{\cong}}\, t^{n}$, for some $n\in I$. The same arguments that we used in the proof of Theorem~\ref{important} for the $q=1$ case, in order to prove that the set $\mathcal{B}_p$ is linearly independent, also apply for the $q>1$ case, and we conclude that the set $\mathcal{B}_p$ forms a basis for KBSM($L(p,q)$).

\begin{remark}\rm
It is possible to use the algebraic method presented in \S~\ref{kblens} to derive the basis $\mathcal{B}_p$ for KBSM($L(p,q)$), $q>1$. Only the calculations involved are more complicated due to the complexity of the $(p, q)$-braid band move: $bbm\left(t^n\right)\, = \, \left[(\sigma_{q-1}\, \ldots\, \sigma_1)\, t\right]^p\, t_{q}^n\, \sigma_q^{\pm 1}$. If we consider $bbm(t^n)$ in KBSM(ST), we may express it as a sum of elements in $B_{\rm ST}$. Using the Kauffman bracket skein relations together with Lemma~\ref{nein} and Proposition~\ref{th1}, one may show that $bbm\left(t^n\right)\, = \, \underset{i=0}{\overset{p+n}{\sum}}\, a_i\, t^i$, for some coefficients $a_i$ and for $n\in \mathbb{N}$. The result then follows similarly to $L(p, 1)$ case.
\end{remark}

\section{Conclusions}

In this paper we introduce the concept of unoriented braids, that play an important role for computing Kauffman bracket skein modules of c.c.o. 3-manifolds, and that seems promising in serving as a tool for studying c.c.o. 3-manifold invariants in general. Then, we introduce the generalized Temperley-Lieb algebra of type B, $TL_{1, n}$, and Jones' original idea (\cite{Jo}), we present the most generic invariant, $V$, for knots and links in the solid torus ST, of the Kauffman bracket type. We then extend $V$ to an invariant for knots and links in the lens spaces by imposing on $V$ relations coming from the (braid) band moves. Our starting point is the Kauffman bracket skein module of ST, KBSM(ST), since 
\[
{\rm KBSM(}L(p,q){\rm)}\, =\, {\rm KBSM(ST)}/<{\rm band\ moves}>.
\]
We first present a new basis for KBSM(ST), $B_{\rm ST}$, using the concept of unoriented braids and we then solve the infinite system of equations $V_{\widehat{a}}\ =\ V_{\widehat{bbm(a)}}$, where $bbm(a)$ denotes the result of a $(p,1)$-braid band move on an element $a$ in the basis of KBSM(ST). This is equivalent to computing the Kauffman bracket skein module of $L(p,1)$ and it is the first time that a skein module is fully computed via this algebraic approach based on braids.
\smallbreak
For the case of the HOMFLYPT skein module of the lens spaces $L(p,1)$, the reader is referred to \cite{DL3, DL4, DLP, D3} for an algebraic approach based on braids and the generalized Hecke algebra of type B, and \cite{GM} for a diagrammatic approach via arrow diagrams (see also \cite{G, DGLM}). We then present a new diagrammatic approach for computing KBSM($L(p,1)$) and we finally extend this approach for computing KBSM($L(p,q)$), $q>1$. Finally, it is worth mentioning that in \cite{D5, D6, D7, D8, D9}, skein modules are also discussed for different families of ``knotted objects'' in 3-manifolds, such as tied links, pseudo and singular links and knotoids.

\end{document}